\documentclass[twoside,a4paper,12pt,centertags]{amsart}
\usepackage{amsmath,amssymb,verbatim,vmargin}
\usepackage[bookmarks=true]{hyperref}   

\newcommand{\norm}{\,|\!|\,}

\newcommand{\loc}{\text{\rm loc}}
\newcommand{\osc}{\text{\rm osc}}

\numberwithin{equation}{section} 
\newtheorem{theorem}{\textbf{THEOREM}}[section]
\newtheorem{lemma}{\textsc{Lemma}}[section]

\newtheorem{corollary}{\textsc{Corollary}}[section]
\theoremstyle{definition}

{\theoremstyle{remark} \newtheorem{remark}{Remark}[section]}

\newcommand{\divo}{\textnormal{div}_{H}}

\def\eqn#1$$#2$${\begin{equation}\label#1#2\end{equation}}

\newcommand{\R}{\mathbb R}

\newcommand{\X}{\mathfrak X}
\newcommand{\Xu}{\mathfrak X u}

\renewcommand{\epsilon}{\varepsilon}

\def\mvint_#1{\mathchoice
          {\mathop{\vrule width 6pt height 3 pt depth -2.5pt
                  \kern -8pt \intop}\nolimits_{\hspace{-1ex}#1}}%
          {\mathop{\vrule width 5pt height 3 pt depth -2.6pt
                  \kern -6pt \intop}\nolimits_{\hspace{-1ex}#1}}%
          {\mathop{\vrule width 5pt height 3 pt depth -2.6pt
                  \kern -6pt \intop}\nolimits_{\hspace{-1ex}#1}}%
          {\mathop{\vrule width 5pt height 3 pt depth -2.6pt
                  \kern -6pt \intop}\nolimits_{\hspace{-1ex}#1}}}

\newcommand{\deltaX}{\big(\delta +|\Xu|^2\big)}

\newcommand{\weight}{\big(\delta+|\Xu|^2\big)^\frac{p-2}{2}}
\newcommand{\weights}{\big(\delta+|\Xu|^2\big)^\frac{p-2+\beta}{2}}

 \newcommand{\intav}{-\hskip -1.1em\int} \newbox\tratto
\setbox\tratto=\hbox{{\count0=0\dimen0=-,9pt\dimen1=1,1pt
\loop\ifnum\count0<11 \advance \count0 by1 \vrule width.51pt
height\dimen1 depth\dimen0\kern-0.17pt \advance\dimen0
by-0.05pt\advance\dimen1 by0.1pt\repeat
\loop\ifnum\count0<21\advance \count0 by1 \vrule width.6pt
height\dimen1 depth\dimen0\kern-0.2pt \advance\dimen0
by-0.1pt\advance\dimen1by0.05pt\repeat}}

\newcommand{\medint}{\displaystyle\copy\tratto\kern-10.4pt\int\limits}

\begin{document}

\title[Variational problems in the Heisenberg group]
{Regularity for variational problems in the Heisenberg group}
\author{Xiao Zhong}
\thanks{Zhong was supported by the Academy of Finland.}
\date{\today}
\address{Department of Mathematics and Statistics, University of Helsinki, P.O. Box 68 (Gustaf H\"allstr\"omin katu 2b)
FI-00014 University of Helsinki, Finland}


\begin{abstract}
We study the regularity of minima of scalar variational
integrals of $p$-growth, $1<p<\infty$, in the Heisenberg group.
\end{abstract}

\maketitle

\begin{center} \textit{In memoriam: Juha Heinonen (1960 - 2007)}\end{center}

\setcounter{tocdepth}{1}
 \tableofcontents

\section{Introduction\label{section:intro}}

In this paper, we study the regularity of minima of scalar variational
integrals in the Heisenberg group ${\mathbb H}^n, n\ge 1$. We
consider the variational problem
\begin{equation}\label{functional}
I(u)=\int_\Omega f(\X u)\, dx. \end{equation} Here $\Omega$ is a
domain in ${\mathbb H}^n$, $u:\Omega\to {\mathbb R}$ is a function
and $\X u=(X_1u, X_2u,\ldots,X_{2n}u)$ its horizontal gradient.
The convex integrand function $f\in C^2({\mathbb R}^{2n};{\mathbb
R})$ is of $p$-growth, $1<p<\infty$. It satisfies the following
growth and ellipticity conditions
\begin{equation}\label{structure}
\begin{aligned}(\delta+\vert z\vert^2)^{\frac{p-2}{2}} \vert\xi\vert^2\le
\langle D^2f(z) &\xi,\xi\rangle \le L(\delta+\vert
z\vert^2)^{\frac{p-2}{2}}\vert \xi\vert^2;\\
\vert Df(z)\vert & \le L(\delta+\vert z\vert^2)^{\frac{p-1}{2}}
\end{aligned}
\end{equation}
for all $z, \xi\in \R^{2n}$, where $\delta\ge 0, L\ge 1$ are
constants.

The natural domain of variational problem (\ref{functional}) is the horizontal
Sobolev space $HW^{1,p}(\Omega)$, see Section
\ref{section:preliminaries} for the definition. By the direct
method in Calculus of Variations, we can easily prove the
existence of minimizers in $HW^{1,p}(\Omega)$ with prescribed boundary value for
functional (\ref{functional}) under the structure condition
(\ref{structure}). The corresponding Euler-Lagrange equation of
(\ref{functional}) is
\begin{equation}\label{equation:main}
\divo \big(Df(\X u)\big)=\sum_{i=1}^{2n}X_i\big(D_i f(\X u)\big)=0.
\end{equation}
where $Df=(D_1 f, D_2f,\ldots,D_{2n}f)$ is the Euclidean gradient
of $f$. It is easy to see that a function in $HW^{1,p}(\Omega)$ is
a local minimizer of (\ref{functional}) if and only if it is a
weak solution of equation (\ref{equation:main}), see Section
\ref{section:preliminaries} for the definitions of local
minimizers and weak solutions.

A typical example of (\ref{functional}) is the following
$p$-energy functional
\begin{equation}\label{penergy}
I(u)=\int_\Omega\big(\delta+\vert \Xu\vert^2\big)^\frac{p}{2}\, dx
\end{equation}
for a constant $\delta\ge 0$. The corresponding Euler-Lagrange
equation is the non-degenerate $p$-Laplacian equation
\begin{equation}\label{equation:p:nondeg}
\divo \big(\weight\X u\big)=0,\end{equation} when $\delta>0$,
and the $p$-Laplacian equation
\begin{equation}\label{equation:p}
\divo\big(\vert\X u\vert^{p-2}\X u\big)=0,
\end{equation}
when $\delta=0$. The weak solutions of equation (\ref{equation:p})
are called $p$-harmonic functions.

There are several interesting cases of values of $p$. When $p=2$,
equation (\ref{equation:p}) is reduced to the Laplacian equation, 
and the solutions are called harmonic functions.
Equation (\ref{equation:p}) is linear when $p=2$; the sum of two
harmonic functions is harmonic. For $p$ different from $2$,
equation (\ref{equation:p}) is not linear; the sum of two
$p$-harmonic functions is not a $p$-harmonic function, in general.

When $p=2n+2$, the Hausdorff dimension of ${\mathbb H}^n$, 
equation (\ref{equation:p}) is tightly connected with the conformal 
mappings, and equation (\ref{equation:main}) the quasiconformal mappings in the Heisenberg group.
We refer these connections and the study of quasiconformal mappings in Carnot groups
to the work of  Kor\`anyi and Reimann \cite{KR1, KR2}, Heinonen and Holopainen \cite{HH}, Capogna \cite{Ca2}, and
Capogna and Cowling \cite{CC}. 
Other two interesting cases are $p=1$ and $p=\infty$, which are excluded in this work. We refer to \cite{CCM1, CCM2, B} and the 
references therein for the study of mean curvature equation and infinite Laplacian equation in Carnot groups.

The regularity theory for equation (\ref{equation:main}) is well
established in the case $p=2$. The study of regularity goes back
to the work of H\"ormander. The classical paper \cite{H} of
H\"ormander treated the linear equation with a general vector
fields. We also mention the remarkable work \cite{F, F2, Kohn} for
the linear equation. We refer to the monograph \cite{BLU} for
the explosive studies and historic notes for the linear equations
in the Heisenberg group, or more generally in the Carnot groups. In
the case $p=2$, when equation (\ref{equation:main}) is not linear,
Capogna obtained the H\"older continuity for the gradient of weak
solutions \cite{Ca1, Ca2}, under the structure condition
(\ref{structure}).

For $p\neq 2$, equation ({\ref{equation:main}) is quasilinear. It
is known that the weak solutions of equation (\ref{equation:main})
are H\"older continuous \cite{CDG}. Concerning the gradient of the
weak solutions, the partial regularity result (regularity outside
a set of measure zero) was obtained in \cite{CG}, see also \cite{Fog}. By the
Cordes perturbation techniques, Domokos and Manfrendi \cite{DM,
DM2} proved the H\"older continuity of the gradient of $p$-harmonic functions when $p$ is
close to $2$. No explicit bound on $p$ was given.

Other regularity results concerning equation (\ref{equation:main})
for $p\neq 2$ include the following ones. Domokos \cite{D} showed
that the vertical derivative $Tu\in L^p_{\loc}(\Omega)$, if $1<p<4$, for the weak solutions $u$
of equation (\ref{equation:main}). He also showed the
integrability of second order horizontal derivatives of $u$. This
extends an earlier result of Marchi \cite{Ma}. The
Lipschitz continuity of $u$ was obtained in \cite{MZZ} for $p$ in
the range $[2,4)$. This result is true not only for the
non-degenerate case ($\delta>0$) but also for the degenerate one
($\delta=0$). This extends an earlier result of Manfredi and
Mingione \cite{MM} concerning the Lipschitz continuity for the
non-degenerate case $(\delta>0)$ with $p$ in a smaller range. Both
of the proofs in \cite{MM,MZZ} use Domokos' result on the
integrability of $Tu$. The restriction on $p, p<4,$ was
unavoidable.

Now a natural question arises: is there a regularity theory for
equation ({\ref{equation:main}) in the Heisenberg group, which is
similar to that in the Euclidean setting? This is the case for
$p=2$. For $p\neq 2$, it is well known that weak solutions of
equations of type (\ref{equation:main}) in the Euclidean spaces
have H\"older continuous derivatives, see \cite{Ur, LU, E, Di, Le, To}.
The sharp H\"older exponent was obtained in \cite{IM} for the
$p$-harmonic functions in the plane. The $C^{1,\alpha}$ regularity
is optimal when $p\ge 2$.

In this paper, we study the regularity of weak solutions of
equation (\ref{equation:main}) both in the non-degenerate case and
in the degenerate case. First, we prove the boundedness of the
horizontal gradient, and hence the Lipschitz continuity of weak
solutions for all $1<p<\infty$. We remark that this result holds both for
the non-degenerate case $\delta>0$ and for the degenerate one $\delta=0$. 

\begin{theorem}\label{thm:lip}
Let $1<p<\infty$, $\delta\ge 0$ and $u\in HW^{1,p}(\Omega)$ be a weak solution of
equation (\ref{equation:main}). Then $\Xu \in
L^\infty_{\loc}(\Omega;{\mathbb R}^{2n})$. Moreover, for any
ball $B_{2r}\subset \Omega$, we have that
\begin{equation}\label{Xu:bdd}
\sup_{B_r}\vert\X u\vert\le c\Big(\intav_{B_{2r}} \deltaX^{\frac p 2}\, dx\Big)^{\frac 1 p},
\end{equation}
where $c>0$ depends only on $n,p,L$.
\end{theorem}

\begin{corollary}\label{cor:lip}
Let $1<p<\infty$, $\delta\ge 0$ and $u\in HW^{1,p}(\Omega)$ be a weak solution of equation (\ref{equation:main}). Then $u$ is locally Lipschitz
continuous in $\Omega$. Moreover, for any ball $B_{2r}\subset
\Omega$, we have that
\begin {equation}\label{Xu:lip}
\vert u(x)-u(y)\vert\le
c\Big(\intav_{B_{2r}}\deltaX^{\frac{ p}{ 2}}\, dz\Big)^{\frac{1}{p}} d(x,y)
\end{equation}
for all $x,y\in B_r$, where $c>0$ depends only on $n,p, L$.
\end{corollary}

Here and in the following, the ball $B_r$ is defined with respect to the 
Carnot-Carath\`eodory metric (CC-metric) $d$; $B_{2r}$ is the double size ball with the same center, see Section \ref{section:preliminaries} for the
definitions.

Second, we show that the horizontal gradient of weak solutions of
equation (\ref{equation:main}) is H\"older continuous when $p\ge
2$.

\begin{theorem}\label{thm:holder}
Let  $2\le p<\infty$, $\delta\ge 0$ and $u\in HW^{1,p}(\Omega)$ be a weak
solution of equation (\ref{equation:main}). Then the horizontal
gradient $\X u$ is H\"older continuous. Moreover, there is a
positive exponent $\alpha=\alpha(n,p,L)\le 1$ such that for any
ball $B_{r_0}\subset \Omega$ and any $0<r\le r_0$, we have
\begin{equation}\label{Xu:holder} \max_{1\le l \le 2n} {\operatorname{osc}}_{B_r} X_l u
\le
c\Big(\frac{r}{r_0}\Big)^\alpha\Big(\intav_{B_{r_0}}\deltaX^{\frac{
p}{ 2}}\, dx\Big)^{\frac{1}{p}},\end{equation} where $c>0$ depends only on
$n,p, L$.
\end{theorem}

This leaves open the H\"older continuity of horizontal gradient of weak solutions for equation
(\ref{equation:main}) in the case $1<p<2$.
Our approach does not work for this case. It seems that it requires new ideas to handle this case.

We comment on our proofs of Theorem \ref{thm:lip} and Theorem \ref{thm:holder}. We first prove the theorems for the case
$\delta>0$, and then for the case $\delta=0$ by an approximation argument. The crucial point is that
in the estimates (\ref{Xu:bdd}) and (\ref{Xu:holder}), the constants $c$ and the exponent $\alpha$ are independent of $\delta$. 
We prove Theorem \ref{thm:lip} in Section \ref{section:lip} and Theorem \ref{thm:holder} in Section \ref{section:Holder}, under the 
supplementary assumption that the solution $u$ is Lipschitz continuous. 
This additional assumption is removed 
by the Hilbert-Haar existence theory for functional (\ref{functional}) in Section \ref{section:approx}.
One good point to make this supplementary assmption is that we 
have enough regularity for the solution $u$ to carry out all of our proofs. Our proofs do not involve the difference quotient.

We use Moser's iteration to prove the Lipschitz continuity of $u$. The essential point is to prove a Caccioppoli type inequality
for $\Xu$ in Theorem \ref{thm:Xu:cacci}. It is an analogous version of that in the setting of Euclidean spaces. 
This is somehow surprising, since 
the vertical derivative $Tu$ is not involved. We should compare it with the usual version, Lemma \ref{caccioppoli:horizantal:sigma},
where $Tu$ is involved. The reason that we can remove the item involving $Tu$ is due to the reverse inequality for $Tu$, obtained
in Lemma \ref{caccioppoli:horizontal:T}. Roughly speaking, Lemma \ref{caccioppoli:horizontal:T} shows that
the vertical derivative $Tu$, comparing with $\X\X u$,  is somehow small. We have a good control on $Tu$. This opens a way to handle
the quasilinear elliptic equations and systems in the Heisenberg group. 




We use De Giorgi's method to prove the H\"older continuity of $\X u$, which is similar to the approach by DiBenedetto \cite{Di}. 
One crucial point is to obtain the Caccioppoli inequality in Lemma \ref{lemma:cacci:k}, which is based on the integrability of $Tu$ in Corollary
\ref{cor:Tu:high} and on the repeating applications of the equation for $Tu$.

The ideas in this paper can be also applied to study the regularity of minima of vectorial variational integrals of the following form
\[ I(u)=\int_\Omega g(\vert \X u\vert)\, dx,\]
where $u:\Omega\to {\mathbb R}^N, N>1,$ is a vector valued function and $g:[0,\infty)\to {\mathbb R}$ is of $p$-growth.

Finally, we remark that the regularity results for the functional (\ref{functional}) in this paper can be applied to study more general functionals like
\[ I(u)=\int_\Omega f(x,u,\X u)\, dx,\]
where $f: {\mathbb R}^{2n+1}\times {\mathbb R}\times {\mathbb R}^{2n}\to {\mathbb R}$ satisfies suitable growth and ellipticity conditions.
We refer to \cite{M} for this kind of treatments in the setting of Euclidean spaces.

\section{Preliminaries}\label{section:preliminaries}
In this section, we fix our notation and introduce the Heisenberg
group ${\mathbb H}^n$ and the sub-elliptic equations.

Throughout this paper, we denote by $c$ a positive constant, that
may vary from line to line. Except explicitly being specified, it
depends only on the dimension $n$ of the Heisenberg group that we
work with, and on the constants $p$ and $L$ in the structure
condition (\ref{structure}). But, it does not depend on $\delta$
in (\ref{structure}).

\subsection{Heisenberg group ${\mathbb H}^n$} We identify the Heisenberg group ${\mathbb H}^n$ with the Euclidean space ${\mathbb
R}^{2n+1}, n\ge 1$. The group multiplication is given by
\[ xy=(x_1+y_1, \dots, x_{2n}+y_{2n}, t+s+\frac{1}{2}
\sum_{i=1}^n (x_iy_{n+i}-x_{n+i}y_i))\] for points
$x=(x_1,\ldots,x_{2n},t), y=(y_1,\ldots,y_{2n},s)\in {\mathbb H}^
n$. The left invariant vector fields corresponding to the
canonical basis of the Lie algebra are
\[ X_i=\partial_{x_i}-\frac{x_{n+i}}{2}\partial_t, \quad
X_{n+i}=\partial_{x_{n+i}}+\frac{x_i}{2}\partial_t,\] and the only
non-trivial commutator
\[ T=\partial_t=[X_i,X_{n+i}]=X_iX_{n+i}-X_{n+i}X_i\]
for $1\le i\le n$. We call $X_1, X_2,\ldots, X_{2n}$ horizontal
vector fields and $T$ the vertical vector field. We denote by
$\X=(X_1, X_2,\ldots,X_{2n})$ the horizontal gradient.
The second horizontal derivatives are given by the horizontal Hessian
$\X \X u$ of a function $u$, with entries $X_i(X_j u), i,j=1,\ldots, 2n$. Note that it is not symmetric, in general.
The standard Euclidean gradient of a function $v$ in ${\mathbb R}^k$ is denoted by
$Dv=(D_1v,\ldots, D_kv)$ and the Hessian matrix by $D^2v$.

The Haar measure in ${\mathbb H}^n$ is the Lebesgue measure of
${\mathbb R}^{2n+1}$. We denote by $\vert E\vert$ the Lebesgue
measure of a measurable set $E\subset {\mathbb H}^n$ and by
\[ \intav_E f\, dx=\frac{1}{\vert E\vert}\int_E f\, dx\]
the average of an integrable function $f$ over set $E$.

In this paper, all of the balls $B_\rho(x)=\{ y\in {\mathbb H}^n: d(y,x)<\rho\}$
are defined with respect to the Carnot-Carath\`eodory metric (CC-metric) $d$.
The CC-distance of two points in ${\mathbb H}^n$ is the length of the shortest
horizontal curve joining them.
The CC-metric is equivalent to the Kor\`anyi metric
\[ d_{{\mathbb H}^n}(y,x)= \norm x^{-1}y\norm_{{\mathbb H}^n}\]
by the Kor\`anyi gauge for $x=(x_1,\ldots,x_{2n}, t)$
\[ \norm x\norm_{{\mathbb H}^n}^2=\sum_{i=1}^{2n} x_i^2+\vert t\vert.\]
Since these two metrics are equivalent, we may state our theorems in Section \ref{section:intro}
by the Kor\`anyi balls $K_\rho(x)=\{ y\in {\mathbb H}^n: d_{{\mathbb H}^n}(y, x)<\rho\}$.

The horizontal Sobolev space $HW^{1,p}(\Omega), 1\le p<\infty, \Omega\subset {\mathbb H}^n,$ consists
of functions $u\in L^p(\Omega)$ such that the horizontal distribution gradient $\X u$ is also in $L^p(\Omega)$.
$HW^{1,p}(\Omega)$ is a Banach space with respect to the norm
\[ \norm u\norm_{HW^{1,p}(\Omega)}=\norm u\norm_{L^p(\Omega)}+\norm \X u\norm_{L^p(\Omega)}.\]
$HW^{1,p}_0(\Omega)$ is the closure of $C^\infty_0(\Omega)$ in $HW^{1,p}(\Omega)$ with this norm.
In an obivous way, we define the local space $HW^{1,p}_{\loc}(\Omega)$.
The following Sobolev imbedding theorem is important for the Moser iteration.

\begin{theorem}\label{thm:sobolev}
Let $1\le q< Q=2n+2$. 
For all $u \in HW^{1,q}_0(B_r)$, $B_r\subset {\mathbb H}^n$, we have
\begin{equation}\label{sobolev}
\Big(\intav_{B_r}\vert u\vert^{\frac{Qq}{Q-q}}\, dx\Big)^{\frac{Q-q}{Qq}}\le c r\Big(\intav_{B_r}\vert \X u\vert^q\, dx\Big)^{\frac 1 q},
\end{equation}
where $c=c(n,q)>0$.
\end{theorem}

\subsection{Sub-elliptic equations}
Suppose that the integrand function $f$ of the functional (\ref{functional}) satisfies 
(\ref{structure}). It is easy to prove that (\ref{structure}) implies 
the strong monotonicity 
\begin{equation}\label{mono}
\langle Df(z)-Df(w), z-w\rangle\ge \frac{1}{\overline L} \big(\delta+\vert z\vert^2+\vert w\vert^2\big)^{\frac{p-2}{2}}\vert z-w\vert^2,
\end{equation}
and therefore the ellipticity condition
\begin{equation}\label{stru}
\langle Df(z),z\rangle\ge \frac{1}{\overline L}\big(\delta+\vert z\vert^2\big)^{\frac{p-2}{2}}\vert z\vert^2-{\overline L}\delta^{\frac{p}{2}},
\end{equation}
for all $z,w\in {\mathbb R}^{2n}$. Here ${\overline L}>0$ is a constant, depending only on $p$ and $L$.

We say a function $u\in HW^{1,p}(\Omega)$ is a weak solution of equation (\ref{equation:main}) if 
\[\int_\Omega \langle Df(\X u), \X\varphi\rangle\, dx=0\]
for all $\varphi\in C_0^\infty(\Omega)$. It is easy to prove that
$u$ is a weak solution of equation (\ref{equation:main}) if and only if it is 
a local minimizer of functional (\ref{functional}), that is,
\[ \int_\Omega f(\Xu)\, dx \le \int_\Omega f(\X u+\X \varphi)\, dx\]
for all $\varphi\in C_0^\infty(\Omega)$. By the strong monotonicity (\ref{mono}), it is easy to prove that
the solution to the following Dirichlet problem is unique
\begin{equation*}
\begin{cases} 
&\divo \big(Df(\X u)\big)=0 \quad \text{ in }\Omega;\\
&u-\phi\in HW^{1,p}_0(\Omega),
\end{cases}
\end{equation*}
where $\phi\in HW^{1,p}(\Omega)$ is given. It is also easy to prove the following comparison principle:
let $u, v\in HW^{1,p}(\Omega)$ be weak solutions of equation (\ref{equation:main}). If
$u\ge v$ on $\partial \Omega$ in the sense of Sobolev, then we have $u\ge v$ a.e. in $\Omega$.

\section{Lipschitz continuity \label{section:lip}}
In this section, we prove Theorem \ref{thm:lip}  for the case $\delta>0$ under the
additional assumption that the weak solution is Lipschitz
continuous. This section has three subsections. We prove several
Caccioppoli type inequalities for the horizontal gradient
and the vertical derivative in the first subsection. The
second subsection contains the main lemma, from which Theorem
\ref{thm:Xu:cacci} follows. The proofs of Theorem
\ref{thm:Xu:cacci} and Theorem \ref{thm:lip} are given in the last
subsection.

Throughout this section, $u\in HW^{1,p}(\Omega)$ is a weak
solution of equation (\ref{equation:main}) satisfying the structure condition (\ref{structure}) with $\delta>0$. We make the following
supplementary assumption: $\X u$ is bounded in $\Omega$, that is,
\begin{equation}\label{supplementary} \norm \X u\norm_{L^\infty(\Omega)}\le M\end{equation}
for a constant $M>0$. We remark here that in Section
\ref{section:approx} we will remove this assumption. Under this additional assumption, it follows
from (\ref{structure}) that $f$ satisfies
\begin{equation}\label{structureprime}
\begin{aligned}
\frac{1}{\nu}\vert \xi\vert^2&\le \langle D^2f(\X u)\xi,\xi\rangle\le
\nu\vert \xi\vert^2;\\
&\vert Df(\X u)\vert\le \nu(1+\vert \X u\vert)
\end{aligned}
\end{equation}
for all $\xi\in {\mathbb R}^{2n}$, where $\nu>0$ is a constant,
depending on $p, L,\delta, M$.

Now, we can apply Capogna's results in \cite{Ca1}. Theorem 1.1 and Theorem 3.1 of \cite{Ca1} show that $\X u$ and $Tu$  are
H\"older continuous in $\Omega$, and that
\begin{equation}\label{apregularity} \X u\in
HW^{1,2}_{\loc}(\Omega;{\mathbb R}^{2n}), \quad Tu\in
HW^{1,2}_{\loc}(\Omega)\cap L^\infty_{\loc}(\Omega).\end{equation}
The above regularity is enough for us to carry out all of the proofs
in this section. We should keep (\ref{apregularity}) in the mind.
We remark here that the constants $M$ and $\nu$ do not enter all of the estimates in this section. Because of this fact, we
are able to remove the supplementary assumption
(\ref{supplementary}) in Section \ref{section:approx}.

\subsection{Caccioppoli type inequalities\label{section:cacci}}

The following two lemmas are straight forward; the proofs are
easy. For the sake of completeness, we give the proofs here.
\begin{lemma}\label{ws:horizontal}
Let $v_l=X_l u, l=1,2,\ldots,n$. Then $v_l$ is a weak solution of
\begin{equation}\label{equation:horizontal}
\sum_{i,j=1}^{2n}X_i\big(D_{j}D_if(\X u)
X_jv_l\big)+\sum_{i=1}^{2n}X_i\big(D_{n+l}D_i f(\X u)Tu\big)+T\big(D_{n+l}f(\X
u)\big)=0;
\end{equation}
Let $v_{n+l}=X_{n+l}u, l=1,2,\ldots,n$. Then $v_{n+l}$ is a weak
solution of
\begin{equation}\label{equation:horizontal2}
\sum_{i,j=1}^{2n}X_i\big(D_{j}D_if(\X u)
X_jv_{n+l}\big)-\sum_{i=1}^{2n}X_i\big(D_{l}D_i f(\X u)Tu\big)-T\big(D_{l}f(\X
u)\big)=0;
\end{equation}
\end{lemma}

\begin{proof}
We only prove (\ref{equation:horizontal}). The proof of
(\ref{equation:horizontal2}) is similar. Let $\varphi\in
C^\infty_0(\Omega)$ and fix $l\in \{1,2,\ldots,n\}$. We use
$X_l\varphi$ as a test-function in (\ref{equation:main}) and
obtain that
\[
\int_\Omega \sum_{i=1}^{2n}D_i f(\Xu) X_iX_l\varphi\, dx=0.
\]
Note that $X_lX_i-X_iX_l=0$ if $i\neq n+l$ and that $X_lX_{n+l}-X_{n+l}X_l=T$. Then integration by parts
yields
\begin{equation}\label{weak1}
\begin{aligned}
0=&\int_\Omega\sum_{i=1}^{2n} D_if(\Xu) X_lX_i\varphi\, dx-\int_\Omega D_{n+l}f(\Xu)T\varphi\, dx\\
=&-\int_\Omega\sum_{i=1}^{2n} X_l(D_i f(\Xu))X_i\varphi\, dx+\int_\Omega T(D_{n+l}f(\Xu))\varphi\, dx,\\
\end{aligned}
\end{equation}
which, together with
\[
\sum_{i=1}^{2n} X_l(D_i f(\Xu))
=\sum_{i,j=1}^{2n} D_jD_i f(\Xu)X_jX_l u
+ \sum_{i=1}^{2n} D_{n+l}D_if(\Xu)Tu,
\]
proves the lemma.
\end{proof}

\begin{lemma}\label{ws:T}
$Tu$ is a weak solution of
\begin{equation}\label{equation:T}
\sum_{i,j=1}^{2n} X_i\big(D_jD_i f(\Xu)X_j (Tu)\big)=0.
\end{equation}
\end{lemma}

\begin{proof}
Let $\varphi\in C^\infty_0(\Omega)$. Then
\begin{equation*}
\begin{aligned}
\int_\Omega \sum_{i,j=1}^{2n}D_jD_if(\X u)X_j(Tu)X_i\varphi\, dx&=
\int_\Omega \sum_{i=1}^{2n}T\big(D_i f(\Xu)\big)X_i\varphi\, dx\\
&=-\int_\Omega \sum_{i=1}^{2n}D_if(Xu)X_i(T\varphi)\, dx=0,
\end{aligned}
\end{equation*}
which proves the lemma. The second equality follows from
integration by parts, and the third from the fact that $u$ is a
weak solution of equation (\ref{equation:main}).
\end{proof}

We need the following Caccioppoli inequality for $Tu$. It was
proved in \cite{MM} by the difference quotient. Since we have the
regularity assumption (\ref{apregularity}), we can prove it
directly without using the difference quotient. The proof is standard and easy.
We provide a proof in the Appendix for the readers'
convenience.
\begin{lemma}\label{caccioppoli:T}
For any $\beta\ge 0$ and  all $\eta\in C^\infty_0(\Omega)$, we
have
\begin{equation*}
\int_\Omega \eta^2\weight \vert Tu\vert^\beta\vert\X Tu\vert^2\,
dx \le \frac{c}{(\beta+1)^2}\int_\Omega \vert\X\eta
\vert^2\weight\vert Tu\vert^{\beta+2}\, dx.
\end{equation*}
where $c=c(n,p,L)>0$.
\end{lemma}

We also need the following Caccioppoli type inequality for $\X u$. Comparing with the one
for $Tu$ in Lemma \ref{caccioppoli:T}, it is much more delicate, mainly due to the non-commutativity of the 
horizontal vector fields $X_i$. When $2\le p<4$, it was proved in Lemma 5.1 of \cite{MZZ}, based on the earlier result in \cite{MM}. 
The proof relies on Domokos' result on the integrability on $Tu$: $Tu\in L^p_{\loc}(\Omega)$.   
The proofs there involve the difference quotient. Again, since we have the
regularity assumption (\ref{apregularity}), we can prove it
directly without using difference quotient. We provide a proof in the Appendix.
\begin{lemma}\label{caccioppoli:horizantal:sigma}
For any $\beta\ge 0$ and all $\eta\in C^\infty_0(\Omega)$, we have
\begin{equation*}
\begin{aligned}
\int_{\Omega}\eta^2\weights\vert\X\X u\vert^2\, dx\le&
c\int_\Omega (\vert \X \eta\vert^2+\vert \eta\vert\vert T\eta\vert)\deltaX^{\frac{p+\beta}{2}}\, dx\\
&+c(\beta+1)^4\int_\Omega\eta^2\weights\vert Tu\vert^2\, dx,
\end{aligned}
\end{equation*}
where $c=c(n,p,L)>0$.
\end{lemma}

\subsection{Main Lemma}
The following lemma gives a reverse type inequality for $Tu$, from which we obtain the integrability
result for $Tu$. 
Eventually, Corollary \ref{cor1} allows us to remove the last integral in
Lemma \ref{caccioppoli:horizantal:sigma} to obtain Theorem \ref{thm:Xu:cacci}. It is crucial for the proof of 
the Lipschitz continuity of $u$. To prove this main lemma, we use a special test function and we invoke the weak formula (\ref{weak1}), instead of 
the equations (\ref{equation:horizontal}) and (\ref{equation:horizontal2}) for $X_lu$. This allows us to apply Lemma \ref{caccioppoli:T} to conclude the proof.

\begin{lemma}\label{caccioppoli:horizontal:T}
For any $\beta\ge 2$ and all non-negative $\eta\in C^\infty_0(\Omega)$, we have
\begin{equation*}
\begin{aligned}
\int_\Omega \eta^{\beta+2}&\weight\vert Tu\vert^{\beta}\vert\X\X u\vert^2\, dx\\
&\le
c(\beta+1)^2\vert\vert \X\eta\vert\vert_{L^\infty}^2\int_\Omega\eta^\beta
\deltaX^{\frac p 2} \vert Tu\vert^{\beta-2}\vert \X\Xu\vert^2\, dx,
\end{aligned}
\end{equation*}
where $c=c(n,p,L)>0$.
\end{lemma}

\begin{proof}
Let $\eta\in C^\infty_0(\Omega)$ be a non-negative cut-off function. Fix $\beta\ge 2$ and
$l\in \{ 1,2,\ldots,n\}$.
Let $\varphi=\eta^{\beta+2}\vert Tu\vert^\beta X_lu$.
We use $\varphi$ as a test
function in (\ref{weak1}).
Note that
\[
X_i\varphi=\eta^{\beta+2} \vert Tu\vert^\beta X_iX_l u+\beta \eta^{\beta+2}\vert Tu\vert^{\beta-2}
TuX_l uX_iTu+(\beta+2)\eta^{\beta+1} X_i\eta\vert Tu\vert^\beta X_l u
\]
and that $X_{n+l} X_l=X_lX_{n+l}-T$. We
obtain that
\begin{equation}\label{weak2}
\begin{aligned}
&\int_{\Omega}\sum_i\eta^{\beta+2}\vert Tu\vert^\beta X_l\big(D_if(\Xu)\big)X_lX_iu\,
dx\\
=&\int_\Omega \eta^{\beta+2} X_l\big(D_{n+l}f(\X u)\big) \vert Tu\vert^\beta Tu\, dx\\
&-(\beta+2)\int_{\Omega}\sum_i\eta^{\beta+1} \vert Tu\vert^\beta X_l \big(D_if(\Xu)\big)X_luX_i\eta\, dx\\
&+\int_{\Omega} \eta^{\beta+2}T\big(D_{n+l}f(\Xu)\big)\vert Tu\vert^\beta X_lu\, dx.\\
&-\beta\int_\Omega \sum_i \eta^{\beta +2} \vert Tu\vert^{\beta-2} Tu X_lu X_l\big(D_if(\X u)\big)X_i Tu\, dx\\
=&I_1+I_2+I_3+I_4.
\end{aligned}
\end{equation}
Here and in the following, all of the sums for $i$ are from $1$ to $2n$. We will estimate both sides of (\ref{weak2}) as follows. For the
left hand side, the structure condition (\ref{structure}) implies that
\[
\int_{\Omega}\sum_i\eta^{\beta+2}\vert Tu\vert^\beta X_l\big(D_if(\Xu)\big)X_lX_iu\,
dx\ge \int_{\Omega} \eta^{\beta+2}\weight\vert Tu\vert^\beta \vert X_l\Xu\vert^2\, dx.
\]

For the right hand side, we will show that
the following estimate is true for each item.
\begin{equation}\label{claim1}
\begin{aligned}
\vert I_k\vert \le & c\tau\int_{\Omega}\eta^{\beta+2}\weight\vert Tu\vert^{\beta}\vert\X\X u\vert^2\, dx\\
&+\frac{c(\beta+1)^2\vert\vert \X\eta\vert\vert_{L^\infty}^2}{\tau}
\int_{\Omega}\eta^{\beta}\deltaX^{\frac p 2}\vert Tu\vert^{\beta-2}\vert \X\X u\vert^2\, dx,
\end{aligned}
\end{equation}
for $k=1,2,3, 4$, where $c=c(n,p,L)>0$ and $\tau>0$ is a constant. By the above estimates for both sides of (\ref{weak2}), we end up with
\begin{equation*}
\begin{aligned}
\int_{\Omega} \eta^{\beta+2}\weight &\vert Tu\vert^\beta \vert X_l\Xu\vert^2\, dx \le  c\tau\int_{\Omega}\eta^{\beta+2}\weight\vert Tu\vert^{\beta}\vert\X\X u\vert^2\, dx\\
&+ \frac{c(\beta+1)^2\vert\vert \X\eta\vert\vert_{L^\infty}^2}{\tau}
\int_{\Omega}\eta^{\beta}\deltaX^{\frac p 2}\vert Tu\vert^{\beta-2}\vert \X \X u\vert^2\, dx.
\end{aligned}
\end{equation*}
The above inequality is true for all $l=1,2,\ldots, n$. Similarly,
we can prove that it is true also for all $l=n+1,\ldots, 2n$. Then we may sum
up these estimates for all $l=1,2,\cdots, 2n$. Now,
by choosing $\tau>0$ small enough,
we complete the proof of the lemma, modulo the proof of (\ref{claim1}).

Now we prove (\ref{claim1}). 
First, we start with $I_4$. By the structure condition (\ref{structure}) and the Cauchy-Schwarz inequality
\begin{equation*}
\begin{aligned}
\vert I_4\vert\le& c\beta \int_\Omega
\eta^{\beta+2}\deltaX^{\frac{p-1}{2}}\vert Tu\vert^{\beta-1}\vert X_l \X u\vert\vert\X Tu\vert\, dx\\
\le & \frac{\tau}{\vert\vert \X\eta\vert\vert_{L^\infty}^2}\int_\Omega \eta^{\beta+4}\weight
\vert Tu\vert^\beta\vert\X Tu\vert^2\,
dx\\
&+\frac{c\beta^2\vert\vert \X\eta\vert\vert_{L^\infty}^2}{\tau}
\int_{\Omega}\eta^{\beta}\deltaX^{\frac p 2}\vert Tu\vert^{\beta-2}\vert X_l \X u\vert^2\, dx.
\end{aligned}
\end{equation*}
We then apply  Lemma \ref{caccioppoli:T} to estimate the first integral in the right hand side. By Lemma \ref{caccioppoli:T}, we have 
\begin{equation}\label{weak3}
\begin{aligned}
&\int_{\Omega}\eta^{\beta+4}\weight
\vert Tu\vert^\beta\vert\X Tu\vert^2\,
dx\\
\le & c\int_{\Omega} \eta^{\beta+2}\vert\X \eta\vert^2\weight\vert
Tu\vert^{\beta+2}\, dx.
\end{aligned}
\end{equation}
Thus, 
\begin{equation}\label{weak3prime}
\begin{aligned}
\vert I_4\vert
\le & c\tau\int_\Omega \eta^{\beta+2}\weight
\vert Tu\vert^{\beta+2}\,
dx\\
&+\frac{c\beta^2\vert\vert \X\eta\vert\vert_{L^\infty}^2}{\tau}
\int_{\Omega}\eta^{\beta}\deltaX^{\frac p 2}\vert Tu\vert^{\beta-2}\vert X_l \X u\vert^2\, dx.
\end{aligned}
\end{equation}
Note that $\vert Tu\vert\le 2\vert \X\X u\vert$. (\ref{weak3prime}) implies that $I_4$ satisfies (\ref{claim1}).

Second, we prove  (\ref{claim1}) for $I_1$.
Integration by parts yields
\begin{equation*}
\begin{aligned}
I_1=&-\int_{\Omega} D_{n+l} f(\Xu)X_l(\eta^{\beta+2}\vert Tu\vert^\beta
Tu)\, dx\\
=&-(\beta+1)\int_{\Omega} \eta^{\beta+2} \vert Tu\vert^\beta D_{n+l}f(\Xu)X_lTu\, dx\\
&-(\beta+2)\int_{\Omega}\eta^{\beta+1}
D_{n+l}f(\Xu) X_l\eta \vert Tu\vert^\beta Tu\, dx
=I_{11}+I_{12}.
\end{aligned}
\end{equation*}
We will show that (\ref{claim1}) holds for both $I_{11}$ and $I_{12}$.
For $I_{11}$, by Young's inequality,
\begin{equation*}
\begin{aligned}
\vert I_{11}\vert \le & c(\beta+1)
\int_{\Omega}\eta^{\beta+2}\deltaX^{\frac{p-1}{2}} \vert Tu\vert^\beta\vert \X Tu\vert\,
dx\\
\le & \frac{\tau}{\vert\vert \X\eta\vert\vert_{L^\infty}^2}\int_{\Omega}\eta^{\beta+4}\weight
\vert Tu\vert^\beta\vert\X Tu\vert^2\,
dx\\
&+\frac{c(\beta+1)^2 \vert\vert \X\eta\vert\vert_{L^\infty}^2}{\tau}
\int_{\Omega}\eta^{\beta}\deltaX^{\frac p 2}\vert Tu\vert^\beta\, dx,
\end{aligned}
\end{equation*}
which, together with (\ref{weak3}) and the fact  $\vert Tu\vert\le 2\vert \X\X u\vert$,  implies that (\ref{claim1}) holds for $I_{11}$.
For $I_{12}$, (\ref{claim1}) follows from
\[
\vert I_{12}\vert \le  c(\beta+2)\int_{\Omega}\eta^{\beta+1}\vert\X\eta\vert\deltaX^{\frac{p-1}{2}}\vert Tu\vert^{\beta+1}\, dx,
\]
and Young's inequality. This proves that $I_1$ satisfies (\ref{claim1}), too.

Third, for $I_2$, we have
\[
I_2\le c(\beta+2)\int_{\Omega}\eta^{\beta+1}\vert\X\eta\vert\deltaX^{\frac{p-1}{2}}\vert Tu\vert^\beta\vert X_l\X u\vert\, dx,
\]
from which, together with Young's inequality and  the fact  $\vert Tu\vert\le 2\vert \X\X u\vert$, (\ref{claim1}) for $I_2$ follows. 

Finally,  $I_3$ has the same bound as that of $I_{11}$.
\[
\vert I_3\vert\le c \int_\Omega \eta^{\beta+2}\deltaX^{\frac{p-1}{2}}\vert Tu  \vert^\beta\vert \X Tu\vert\, dx.
\]
Thus, $I_3$ satisfies (\ref{claim1}), too. This completes the proof of (\ref{claim1}), and hence that of the lemma.
\end{proof}



By H\"older's inequality, the following corollary is an easy consequence of Lemma \ref{caccioppoli:horizontal:T}.
\begin{corollary}\label{cor1}
For any $\beta\ge 2$ and all non-negative $\eta\in C^\infty_0(\Omega)$, we have
\begin{equation*}
\begin{aligned}
\int_\Omega \eta^{\beta+2}&\weight\vert Tu\vert^\beta\vert \X\X u\vert^2\, dx\\
&\le c^{\frac{\beta}{2}}(\beta+1)^{\beta}\vert\vert\X\eta\vert\vert_{L^\infty}^\beta
\int_\Omega \eta^2\deltaX^{\frac {p-2+\beta} {2}}\vert\X\X u\vert^2\, dx,
\end{aligned}
\end{equation*}
where $c=c(n,p,L)>0$.
\end{corollary}

\subsection{Proof of Theorem \ref{thm:lip}}
We will prove the following surprising Caccioppoli type inequality for $\X u$, from which the Lipschitz continuity of $u$ follows
by the well-known Moser iteration. It is similar to that for weak solutions of Riemannian
elliptic equations. The following theorem is an easy consequence of  Lemma \ref{caccioppoli:horizantal:sigma} and Corollary \ref{cor1}.

\begin{theorem}\label{thm:Xu:cacci}
Let  $1<p<\infty$. Then for any $\beta\ge 2$ and
all non-negative $\eta\in C^\infty_0(\Omega)$, we have that
\begin{equation}\label{Xu:good}
\int_\Omega \eta^2\weights\vert\X \X u\vert^2\, dx\le c(\beta+1)^{10}K\int_{spt(\eta)} \deltaX^{\frac{p+\beta}{2}}\, dx,
\end{equation}
where $K=\vert\vert \X\eta\vert\vert_{L^\infty}^2+\vert\vert\eta
T\eta\vert\vert_{L^\infty}$ and $c>0$ depends only on $n,p,L$.
\end{theorem}

\begin{proof}
Our goal is to prove (\ref{Xu:good}). By Lemma  \ref{caccioppoli:horizantal:sigma}, we only need to estimate
the integral $\int_\Omega \eta^2\weights \vert Tu\vert^2\, dx$. To this end, by H\"older's inequality,
\begin{equation*}
\begin{aligned}
&\int_\Omega\eta^2\weights\vert Tu\vert^2\, dx\\
\le& \left(\int_\Omega \eta^{\beta+2}\weight\vert Tu\vert^{\beta+2}\, dx\right)^{\frac{2}{\beta+2}}
\left(\int_{spt(\eta)}\deltaX^{\frac{p+\beta}{2}}\, dx\right)^{\frac{\beta}{\beta+2}}.
\end{aligned}
\end{equation*}
Note that $\vert Tu\vert\le 2 \vert\X\X u\vert$. We can continue
to estimate to first integral in the right hand side by Corollary
\ref{cor1}. Then plugging this estimate to the inequality in Lemma
\ref{caccioppoli:horizantal:sigma}, we obtain by Young's inequality
\[
\int_\Omega \eta^2\weights\vert\X \X u\vert^2\, dx
\le c(\beta+1)^{\frac{4(\beta+2)}{\beta}+2} K\int_{spt(\eta)} \deltaX^{\frac{p+\beta}{2}}\, dx,
\]
where $K=\vert\vert \X\eta\vert\vert_{L^\infty}^2+\vert\vert\eta T\eta\vert\vert_{L^\infty}$. This proves the theorem.
\end{proof}

Combining Corollary \ref{cor1} and Theorem \ref{thm:Xu:cacci}, we
obtain the following estimate for $Tu$, which is critical for the
proof of the H\"older continuity of the horizontal gradient of
solutions in Section \ref{section:Holder}.

\begin{corollary}\label{cor:Tu:high}
For any $\beta\ge 2$ and all non-negative $\eta\in C^\infty_0(\Omega)$, we have that
\[
\int_\Omega\eta^{\beta+2}\weight\vert Tu\vert^{\beta+2}\le c(\beta)K^{\frac{\beta+2}{2}}
\int_{spt(\eta)}\deltaX^{\frac {p+\beta} {2}}\, dx,
\]
where $K=\vert\vert \X\eta\vert\vert_{L^\infty}^2+\vert\vert\eta
T\eta\vert\vert_{L^\infty}$ and  $c(\beta)>0$ depends on $n,p,L$
and $\beta$.
\end{corollary}

Now we give the proof of Theorem \ref{thm:lip}.

\begin{proof}[Proof of Theorem \ref{thm:lip}]  Theorem \ref{thm:lip} follows from Theorem \ref{thm:Xu:cacci} by Moser's iteration. 
The proof is the same as that in the setting of Euclidean spaces, see e.g. proof of Theorem
3.34 in \cite{HKM}. We give the outline here.  The Caccioppoli inequality (\ref{Xu:good})
and the Sobolev inequality (\ref{sobolev}) yield
\begin{equation}\label{moser}
\left(\int_\Omega \big(\delta+\vert\X u\vert^2\big)^{\frac{p+\beta}{2}\kappa}\eta^{2\kappa}\, dx\right)^{\frac{1}{\kappa}}\le c(p+\beta)^{12}K\int_{spt(\eta)} \deltaX^{\frac{p+\beta}{2}}\, dx,
\end{equation}
for all $\beta\ge 2$ and for non-negative $\eta\in C^\infty_0(\Omega)$, where $\kappa=Q/(Q-2)=(n+1)/n$, $c=c(n,p,L)>0$ and $K=\vert\vert \X\eta\vert\vert_{L^\infty}^2+\vert\vert\eta
T\eta\vert\vert_{L^\infty}$. Let $B_r\subset \Omega$ and $0<\sigma<1$ be fixed. We define 
\[ r_i=\sigma r+\frac{(1-\sigma)r}{2^{i}}, \quad \beta_i=(p+2)\kappa^i-p, \quad\text{ for } i=0,1,\ldots.\]
By choosing a standard cut-off function $\eta\in C_0^\infty(B_{r_i})$ with $\eta=1$ in $B_{r_{i+1}}$ and letting $\beta=\beta_i$ in (\ref{moser}), we obtain 
\begin{equation}\label{moser2}
\begin{aligned}
&\left(\intav_{B_{r_{i+1}}}\big(\delta+\vert \X u\vert^2\big)^{\frac{\alpha_{i+1}}{2}}\, dx\right)^{\frac{1}{\alpha_{i+1}}}\\
\le & c^{\frac{1}{\alpha_i}}\alpha_i^{\frac{12}{\alpha_i}}2^{\frac{i}{\alpha_i}}
(1-\sigma)^{-\frac{2}{\alpha_i}}\left(\intav_{B_{r_i}}\big(\delta+\vert \X u\vert^2\big)^{\frac{\alpha_{i}}{2}}\, dx\right)^{\frac{1}{\alpha_{i}}},
\end{aligned}
\end{equation}
where $\alpha_i=(p+2)\kappa^i$. Iterating the above inequality, we end up with
\begin{equation}\label{lip100}
\sup_{B_{\sigma r}} \big(\delta+\vert \X u\vert^2\big)^{\frac{1}{2}}\le c(1-\sigma)^{-\frac{Q}{p+2}}\left(\intav_{B_{r}} \deltaX^{\frac {p+2} {2}}\, dx\right)^{\frac {1} {p+2}},
\end{equation}
for all $B_r\subset \Omega$ and all $0<\sigma<1$, where $c=c(n,p,L)>0$. 
Now the estimate (\ref{lip100}) hold for all balls $B_r\subset \Omega$ and all $0<\sigma<1$.  Another iteration argument (see the proof of Lemma 3.38 in \cite{HKM}) shows that
\begin{equation}\label{lip101}
\sup_{B_{\sigma r}} \big(\delta+ \vert \X u\vert^2\big)^{\frac{1}{2}}\le c(1-\sigma)^{-\frac{Q}{q}}\left(\intav_{B_{r}} \deltaX^{\frac {q} {2}}\, dx\right)^{\frac {1} {q}},
\end{equation}
for any $q>0$, where $c=c(n,p,L,q)>0$. We may let $q=p$. This completes the proof of Theorem \ref{thm:lip}.
\end{proof}

\section{H\"older continuity of the horizontal gradient\label{section:Holder}}

In this section, we assume that $p\ge 2$, and we will prove the
H\"older continuity of the horizontal gradient of solutions of
equation (\ref{equation:main}) for this range of $p$, under the supplementary assumption (\ref{supplementary}). This section
is divided into three subsections. First, we introduce the De
Giorgi's class of functions in the setting of Heisenberg group in
subsection \ref{subsection:DeGoirgi:class}. Second, we show that
the gradient of solutions of equation (\ref{equation:main}) satisfies a Caccioppoli inequality in subsection
\ref{subsection:cacci}, and finally,  we prove Theorem
\ref{thm:holder} in subsection \ref{subsection:proof}.

\subsection{De Giorgi's class of functions}\label{subsection:DeGoirgi:class}

The fundamental work of De Giorgi \cite{De} showed the local
boundedness and H\"older continuity for functions satisfying
certain integral inequalities, nowadays known as De Giorgi's class
of functions. In this section, we consider this class of functions
defined in the Heisenberg group.

Let $B_{\rho_0}\subset {\mathbb H}^n$ be a ball and $k_0\in {\mathbb R}$ be a constant. The De Giorgi's
class $DG^+(B_{\rho_0})$ consists of functions $v\in
HW^{1,2}(B_{\rho_0})\cap L^\infty(B_{\rho_0})$, which satisfy for any
balls $B_{\rho^\prime},B_\rho$ with the same center as
$B_{\rho_0}$ and $0<\rho^\prime<\rho\le \rho_0$, and for any $k\in {\mathbb R}$, the following inequality
\begin{equation}\label{DG+}
\int_{B_{\rho^\prime}}\vert\X (v-k)^+\vert^2\, dx\le
\frac{\gamma}{(\rho-\rho^\prime)^2}
\int_{B_\rho}\vert(v-k)^+\vert^2\, dx+\chi^2\vert
A_{k,\rho}^+\vert^{1-\frac{2}{q}},
\end{equation}
where $ A_{k,\rho}^+=\{ x\in B_\rho: (v(x)-k)^+=\max (v(x)-k,0)>0\}$.
The parameters $\gamma,\chi$ are arbitrary non-negative numbers. We
require that $q>Q=2n+2$, the Hausdorff dimension of ${\mathbb
H}^n$. We do not exclude the case $q=\infty$. Similarly, we define
the class $DG^-(B_{\rho_0})$. If $v\in DG^+(B_{\rho_0})$, then
$-v\in DG^-(B_{\rho_0})$. We denote
$DG(B_{\rho_0})=DG^+(B_{\rho_0})\cap DG^-(B_{\rho_0})$.

The following lemma is true for a bigger class of functions than
$DG^+(B_{\rho_0})$. We denote by $SDG^+(B_{\rho_0})$ the class of
functions $v\in HW^{1,2}(B_{\rho_0})\cap L^\infty(B_{\rho_0})$ satisfying
\begin{equation}\label{SDG+}
\int_{B_{\rho^\prime}}\vert\X (v-k)^+\vert^2\, dx\le \frac{\gamma
M(\rho_0)^2}{(\rho-\rho^\prime)^2}\vert A_{k,\rho}^+\vert
+\chi^2\vert A_{k,\rho}^+\vert^{1-\frac{2}{q}},
\end{equation}
for any $k\in {\mathbb R}$, and for all balls
$B_{\rho^\prime},B_\rho$ with the same center as $B_{\rho_0}$ and
$0<\rho^\prime<\rho\le \rho_0$. Here $M(\rho_0)$ is a positive
constant. Note that (\ref{DG+}) implies (\ref{SDG+}) with
$M(\rho_0)=2\sup_{B_{\rho_0}} \vert v\vert$.

The proof of the following lemma is similar to that of Lemma 6.1
of \cite{LU}. We omit the proof.
\begin{lemma}\label{DGvanish}
For any $b\in (0,1)$,  there exists $\theta=\theta(\gamma,n,q,
b)>0$ such that the following holds: for any function $v\in
SDG^+(B_{\rho_0})$, and for any number $k_0$, the inequality
\[ \vert A_{k_0,\rho_0}^+\vert\le \theta \rho_0^Q\]
implies that
\[ \sup_{B_{\rho_0/2}} v\le k_0+bM(\rho_0),\]
provided that
\[ M(\rho_0)\ge \chi \rho_0^{1-\frac{Q}{q}}.\]
\end{lemma}

The proof of the following lemma is similar to that of Lemma 6.2
of \cite{LU}. We need only minor modifications. We omit the proof.

\begin{lemma}\label{DGosc}
For any $\tau>0$, there exists $s=s(\gamma,n,q,\tau)>0$ such that
the following holds: for any function $v\in DG^+(B_{\rho_0})$ and
for any number $k_0$, the inequality
\[ \vert A_{k_0,\rho_0/2}^-\vert \ge \tau \rho_0^Q\]
implies that
\[ \sup_{B_{\rho_0/4}} v\le \sup_{B_{\rho_0}}
v-2^{-s}H+\chi\rho_0^{1-\frac{Q}{q}},\]
where
$H=\sup_{B_{\rho_0}}v-k_0$.
\end{lemma}

Here we remark that we can relax the assumption $v\in DG^+(B_{\rho_0})$
in Lemma \ref{DGosc}. We only need to assume that $v$ satisfies inequality
(\ref{DG+}) for all $k\ge k_0$, where $k_0$ is as in Lemma \ref{DGosc}. 
The reason is that in the proof of Lemma \ref{DGosc} we only apply 
inequality (\ref{DG+}) for $k\ge k_0$. 

The following theorem follows easily from Lemma \ref{DGosc}. From
it follows the H\"older continuity of functions in the De Giorgi
class $DG(B_{\rho_0})$.

\begin{theorem}\label{DGHolder}
There exists $s_0=s_0(\gamma,n,q)>0$ such that for any function
$v\in DG(B_{\rho_0})$, we have
\[ \osc_{B_{\rho_0/2}} v\le (1-2^{-s_0})\osc_{B_{\rho_0}}
v+\chi\rho_0^{1-\frac{Q}{q}}.\]
\end{theorem}

\subsection{Caccioppoli inequality}\label{subsection:cacci} In this section, we
assume that $p\ge 2$. Let $u\in HW^{1,p}(\Omega)$ be a solution of
equation (\ref{equation:main}) satisfying the structure condition (\ref{structure}) with $\delta>0$. 
As in Section \ref{section:lip}, we make the supplementary assumption (\ref{supplementary}) that 
$u$ is Lipschitz continuous in $\Omega$. Thus we have the regularity (\ref{apregularity}) for $u$. Keeping this in our mind,
we will show that $X_l u,
l=1,2,\ldots, 2n$, satisfies the Caccioppoli inequality
(\ref{cacci:Xu:k}) in the following lemma. The proof of the lemma
is based on the estimate of $Tu$ in Corollary \ref{cor:Tu:high}. It
also involves an iteration argument.

We fix a ball $B_{r_0}\subset \Omega$. We denote by
\[ \mu(r_0)=\max_{1\le l\le 2n}\sup_{B_{r_0}} \vert X_l u\vert.\]
Consider the balls $B_{r^\prime}, B_r$ with the same center as
$B_{r_0}$ and $0<r^\prime<r\le r_0/2$. Denote as before
$A^+_{k,r}=\{ x\in B_r: (u(x)-k)^+>0\}$.

\begin{lemma}\label{lemma:cacci:k}
For any $q\ge 4$, there exists $c=c(n,p,L,q)>0$ such that the
following inequality holds for any $1\le l\le 2n$, for any $k\in {\mathbb R}$ and for any
$0<r^\prime<r\le r_0/2$
\begin{equation}\label{cacci:Xu:k}
\begin{aligned}
&\int_{A^+_{k,r^\prime}} \weight\vert \X X_l u\vert^2\, dx\\
\le&
\frac{c}{(r-r^\prime)^2}\int_{A^+_{k,r}}\weight\vert(X_lu-k)^+\vert^2\,
dx +c K\vert A^+_{k,r}\vert^{1-\frac{2} {q}}
\end{aligned}
\end{equation}
where $K =r_0^{-2}\vert
B_{r_0}\vert^{\frac 2 q}(\delta+\mu(r_0)^2)^{\frac{p} {2}}$.
\end{lemma}

\begin{proof} We only prove the lemma for $l=1,2,\cdots,n$; we can prove similarly the lemma 
for $l=n+1,n+2,\cdots, 2n$. We observe that we only need to prove 
(\ref{cacci:Xu:k}) for $k$ such that $\vert k\vert \le \mu(r_0)$.
Now fix $l\in \{1,2,\ldots,n\}$, and $k\in {\mathbb R}$ such that
$\vert k\vert \le \mu(r_0)$ and fix $0<r^\prime<r\le
r_0/2$. Let $\varphi=\eta^2 v$, where $v=(X_lu-k)^+$ and $\eta\in
C^\infty_0(B_r)$ is a cut-off function such that $\eta=1$ in $B_{r^\prime}$ and $\vert \X \eta\vert\le 2/(r-r^\prime)$. We use $\varphi$ as a
test-function in (\ref{equation:horizontal}) to obtain that
\begin{equation*}
\begin{aligned}
\int_{B_r} \sum_{i,j}&\eta^2D_jD_i f(\Xu)X_jX_l uX_iv\, dx\\
=&-2\int_{B_r}\sum_{i,j}\eta v D_jD_i f(\Xu)X_jX_l uX_i \eta\, dx\\
&-\int_{B_r} \sum_i D_{n+l}D_i f(\Xu)TuX_i(\eta^2 v)\, dx\\
&+\int_{B_r} \eta^2 T(D_{n+l}f(\Xu)) v\, dx
\end{aligned}
\end{equation*}
By Young's inequality and the structure condition
(\ref{structure}), we arrive at
\begin{equation}\label{estadd1}
\begin{aligned}
\int_{B_r}\eta^2\weight\vert \X v\vert^2\, dx \le &
c\int_{B_r}\vert
\X\eta\vert^2\weight v^2\, dx\\
&+c\int_{ A^+_{k,r}}\eta^2 \weight\vert Tu\vert^2\, dx\\
&+c\int_{B_r}\eta^2\weight\vert\X Tu\vert v\, dx\\
&=I_1+I_2+I_3.
\end{aligned}
\end{equation}
Observe that to prove lemma, we only need to estimate $I_2$ and
$I_3$. We estimate $I_2$ as follows. By H\"older's inequality,
\[ I_2\le \left(\int_{B_{r_0/2}} \weight\vert Tu\vert^q\,
dx\right)^{\frac 2 q}\left(\int_{A_{k,r}^+}\weight\,
dx\right)^{1-\frac{2}{q}}.\] 
Recall that Corollary
\ref{cor:Tu:high} gives
\[ \int_{B_{r_0/2}}\weight\vert Tu\vert^q\, dx\le
cr_0^{-q}\int_{B_{r_0}}\big(\delta+\vert\X
u\vert^2\big)^{\frac{p-2+q}{2}}\, dx,\] 
where $c=c(n,p,L,q)>0$. Note that $p\ge 2$. Thus $I_2$ is bounded from above by
the last term of inequality (\ref{cacci:Xu:k}), 
\begin{equation}\label{estadd2} I_2\le c r_0^{-2}\vert
B_{r_0}\vert^{\frac 2 q}(\delta+\mu(r_0)^2)^{\frac{p} {2}}\vert
A^+_{k,r}\vert^{1-\frac 2 q}=cK\vert A_{k,r}^+\vert^{1-\frac{2}{q}},
\end{equation} 
where $c=c(n,p,L,q)>0$. 

The estimate
of $I_3$ is involved. We will show that
\begin{equation}\label{estadd3}
 I_3 \le \frac{ M}{2}+cK\vert A_{k,r}^+\vert^{1-\frac{2}{q}},
\end{equation}
where $c=c(n,p,L,q)>0$ and
\begin{equation}\label{addM}
M=\int_{B_r}\eta^2\weight\vert \X v\vert^2\, dx+\int_{B_r}\vert\X \eta\vert^2\weight v^2\, dx. 
\end{equation}

Now the estimates (\ref{estadd1}), (\ref{estadd2}) and (\ref{estadd3})
give us
\[ \int_{B_r}\eta^2\weight\vert \X v\vert^2\, dx \le 
c\int_{B_r}\vert
\X\eta\vert^2\weight v^2\, dx+cK\vert A_{k,r}^+\vert^{1-\frac{2}{q}},\]
from which (\ref{cacci:Xu:k}) follows. This proves the lemma, modulo the proof of (\ref{estadd3}).

Now we prove (\ref{estadd3}). First, for $\beta\ge 0$, let
$\tilde\varphi=\eta^2 v^2\vert Tu\vert^\beta Tu$, where $\eta$ is
as before. We test equation (\ref{equation:T}) with
$\tilde\varphi$ to obtain that
\begin{equation*}
\begin{aligned}
&(\beta+1)\int_{B_r}\eta^2\weight v^2\vert Tu\vert^\beta\vert \X Tu\vert^2\, dx\\
\le & c\int_{B_r}\eta\vert\X \eta\vert\weight v^2\vert Tu\vert^{\beta+1}\vert\X Tu\vert\, dx\\
&+c\int_{B_r}\eta^2\weight v\vert Tu\vert^{\beta+1}\vert\X v\vert\vert\X Tu\vert\, dx,
\end{aligned}
\end{equation*}
where $c=c(n,p,L,q)>0$.
By Young's inequality, we obtained that
\begin{equation}\label{Tu:iteration}
\begin{aligned}
\int_{B_r}&\eta^2 \weight v^2  \vert Tu\vert^\beta\vert\X Tu\vert^2\, dx\\
& \le cM^{\frac{1}{2}}\left(\int_{B_r}\eta^2\weight v^2\vert Tu\vert^{2\beta+2}\vert\X Tu\vert^2\, dx\right)^{\frac 1 2},
\end{aligned}
\end{equation}
where $M$ is defined as in (\ref{addM}).

Second, we iterate (\ref{Tu:iteration}) as follows.
Let $\beta_m=2^m-2, m=1,2,\ldots$. Set
\[a_m=\int_{B_r}\eta^2\weight v^2\vert Tu\vert^{\beta_m}\vert\X Tu\vert^2\, dx.\]
Then (\ref{Tu:iteration}) gives us
\begin{equation}\label{est:a1}
a_1\le cM^{\frac{1}{2}} a_2^{\frac 1 2}\le \cdots\le (cM^{\frac{1}{2}})^{\sum_{i=0}^{m-1} 2^{-i}}a_{m+1}^{2^{-m}}=
(c^2M)^{1-2^{-m}} a_{m+1}^{2^{-m}}.
\end{equation}
We then estimate $a_{m+1}$
as follows. Note that $\vert k\vert\le \mu(r_0)$. Thus
\[ a_{m+1}\le c\mu(r_0)^2\int_{B_{r_0/2}}
\weight \vert Tu\vert^{\beta_{m+1}}\vert\X Tu\vert^2\, dx.\]
We continue to estimate by Lemma
\ref{caccioppoli:T} and Corollary \ref{cor:Tu:high} the integral in the right hand side. We have that
\begin{equation*}
\begin{aligned}
\int_{B_{r_0/2}}\weight \vert Tu\vert^{\beta_{m+1}}\vert\X Tu\vert^2\, dx
\le & cr_0^{-2}\int_{B_{3r_0/4}}\weight\vert Tu\vert^{\beta_{m+1}+2}\, dx\\
\le & c(m) r_0^{-\beta_{m+1}-4}\int_{B_{r_0}}\big(\delta+\vert\X u\vert^2\big)^{\frac{p+\beta_{m+1}}{2}},
\end{aligned}
\end{equation*}
where $c(m)>0$ depends not only on $n,p,L$, but also on $m$. 
Thus we arrive at the following estimate for $a_{m+1}$.
\begin{equation}\label{est:am}
a_{m+1}\le c(m)(\delta+\mu(r_0)^2)^{\frac{p+2^{m+1}}{2}}\vert B_{r_0}\vert r_0^{-2^{m+1}-2}.
\end{equation}

Finally, the estimate (\ref{estadd3}) for $I_3$ follows from the ones (\ref{est:a1}), (\ref{est:am}) for $a_1$ and $a_{m+1}$.
By H\"older's inequality,
\begin{equation*}
\begin{aligned}
I_3 \le & c\left(\int_{B_r}\eta^2\weight\vert\X Tu\vert^2 v^2\, dx\right)^{\frac 1 2}
\left(\int_{A_{k,r}^+} \eta^2\weight\, dx\right)^{\frac 1 2}\\
\le & c a_1^{\frac 1 2}(\delta+\mu(r_0)^2)^{\frac{p-2}{4}}\vert A_{k,r}^+\vert^{\frac 1 2}.
\end{aligned}
\end{equation*}
Here we also used the fact that $p\ge 2$. 
Now combining the above estimate with (\ref{est:a1}), we arrive at
\begin{equation*}
\begin{aligned}
I_3\le &c(m)(\delta+\mu(r_0)^2)^{\frac{p-2}{4}}\vert A_{k,r}^+\vert^{\frac{1}{2}}a_{m+1}^{2^{-m-1}}M^{\frac{1}{2}(1-2^{-m})}\\
\le & \frac{M}{2}+c(m)(\delta+\mu(r_0)^2)^{\frac{p-2}{2}\frac{2^m}{2^m+1}}
\vert A_{k,r}^+\vert^{\frac{2^m}{2^m+1}}a_{m+1}^{\frac{1}{2^m+1}},
\end{aligned}
\end{equation*}
where in the second inequality we used Young's inequality. Then we plug 
the estimate (\ref{est:am}) to the above inequality, and we obtain that
\begin{equation}\label{est:a1am} I_3
\le \frac{M}{2}+c(m)(\delta+\mu(r_0)^2)^{\frac p 2} \vert
B_{r_0}\vert^{\frac{1}{2^m+1}}r_0^{-2}\vert
A_{k,r}^+\vert^{\frac{2^m}{2^m+1}}.\end{equation} Now we may
choose $m$ big enough such that
\[\frac{2^m}{2^m+1}\ge 1-\frac{2}{q}.\]
Then (\ref{est:a1am}) becomes
\[ I_3\le \frac{M}{2}+c K \vert A_{k,r}^+\vert^{1-\frac{2}{q}},\]
where $c=c(n,p,L,q)>0$. This proves (\ref{estadd3}).
We finished the proof of  (\ref{cacci:Xu:k}).
\end{proof}

\begin{remark}\label{remarkadd}
Similarly, 
we can obtain the corresponding inequality (\ref{cacci:Xu:k}) 
for $(X_lu-k)^-$ with 
$(X_lu-k)^+$ replaced by $(X_lu-k)^-$ and 
$A_{k,r}^+$ replaced by $A^-_{k,r}$. 
\end{remark}

\subsection{Proof of Theorem
\ref{thm:holder}}\label{subsection:proof}

In the subsection, we prove Theorem \ref{thm:holder} for the case $\delta>0$.
Let $u\in HW^{1,p}(\Omega)$ be a weak solution of equation (\ref{equation:main}). We assume (\ref{supplementary}).
We fix a ball $B_{r_0}\subset\Omega$. For all balls
$B_r, r\le r_0$, with the same center as $B_{r_0}$, we denote for $l=1,\ldots,2n$,
\[ \mu_l(r)=\sup_{B_r} \vert X_l u\vert,\quad \mu(r)=\max_{1\le l\le 2n} \mu_l(r),\]
and 
\[ \omega_l(r)=\osc_{B_r} X_l u, \quad \omega(r)=\max_{1\le l\le 2n} \omega_l(r).\]

Theorem \ref{thm:holder}, under the additional assumption (\ref{supplementary}), follows easily from Theorem \ref{thm:holder2} below by an iteration argument.

\begin{theorem}\label{thm:holder2}
There exists a constant $s=s(n,p,L)\ge
1$ such that for any $0<r\le
r_0/16$, we have
\begin{equation}\label{mur4r}
\omega(r)\le (1-2^{-s})\omega(8r)+2^s(\delta+\mu(r_0)^2)^{\frac 1
2}\left(\frac{r}{r_0}\right)^{\frac 1 p}.
\end{equation}
\end{theorem}

Actually, the power $1/p$ in (\ref{mur4r}) is not essential; any
number in the interval $(0,2/p)$ will do. In the remaining of this subsection,
we prove Theorem \ref{thm:holder2}. During the proof, we will
specify a fixed and finite number of lower bounds for $s$,
required for the proof to work. These lower bounds depend only on
$n,p,L$. We then take $s$ to be equal to the maximum of this
finite collections of lower bounds. We fix a ball $B_r$ with the
same center as $B_{r_0}$ and with $0<r\le r_0/16$. To prove Theorem
\ref{thm:holder2}, note that we may assume
\begin{equation}\label{assum:mu}
\omega(r)\ge (\delta+\mu(r_0)^2)^{\frac 1 2}\left(\frac{r}{r_0}\right)^{\frac 1 p},
\end{equation}
since, otherwise Theorem \ref{thm:holder2} is true with $s=1$ and we are done. In the following, we assume that
(\ref{assum:mu}) is true. We divide the proof into two cases.

{\it Case 1.} We assume that there exists a number $\tau>0$, small and depending only on $n,p,L$, such that
for at least one $l\in \{ 1,2,\ldots,2n\}$, we have either
\begin{equation}\label{assum:tau}
\vert\{ x\in B_{4r}: X_l u>-\frac{1}{8}\mu(8r)\} \vert\le \tau
r^Q,
\end{equation}
or
\begin{equation}\label{assum:tau2}
\vert\{ x\in B_{4r}: X_l u<\frac{1}{8}\mu(8r)\}\vert \le \tau r^Q.
\end{equation}
The constant $\tau$ will be determined in the following lemma.

\begin{lemma}\label{lemma:tau}
There exists a number $\tau=\tau(n,p,L)>0$ such that if
(\ref{assum:tau}) holds for an index $l$, then
\[\sup_{B_{2r}} X_l u\le -\frac{1}{16}\mu(8r).
\]
Analogously, if (\ref{assum:tau2}) holds for an index $l$, then
\[ \inf_{B_{2r}} X_lu\ge \frac{1}{16}\mu(8r).\]
\end{lemma}

We postpone the proof of Lemma \ref{lemma:tau} and continue the proof of Theorem \ref{thm:holder2}. In each case
of conclusions of Lemma \ref{lemma:tau}, we have for all $x\in B_{2r}$
\begin{equation}\label{Xu:uniform}
c_1(\delta+\mu(8r)^2)^{\frac{p-2}{2}} \le (\delta+\vert\X
u(x)\vert^2)^{\frac{p-2}{2}}\le
c_2(\delta+\mu(8r)^2)^{\frac{p-2}{2}},
\end{equation}
for some constants $0<c_1< c_2$, depending only on $n,p$. We will
apply Lemma \ref{lemma:cacci:k} with $q=2Q$. Due to
(\ref{Xu:uniform}), (\ref{cacci:Xu:k}) becomes
\begin{equation*}
\begin{aligned}
\int_{A^+_{k,r^{\prime\prime}}}\vert\X X_j u \vert^2\, dx \le &
\frac{c}{(r^\prime-r^{\prime\prime})^2}\int_{A_{k,r^\prime}^+}\vert
(X_j u-k)^+\vert^2\, dx\\
&+c K(\delta+\mu(8r)^2)^{\frac{2-p}{2}}\vert A_{k,
r^\prime}^+\vert^{1-\frac{1}{Q}},
\end{aligned}
\end{equation*}
for all $j\in \{1,2,\ldots,2n\}$, and all
$0<r^{\prime\prime}<r^\prime\le 2r$, where $c=c(n,p,L)>0$ and $K$
is as in Lemma \ref{lemma:cacci:k}. We remark here that the above
inequality is true for all $j\in \{1,2,\ldots,2n\}$. Now, for each
$j$, $X_ju$ belongs to $DG^+(B_{2r})$ with $q=2Q$ and $\chi=c
K^{\frac 1 2}(\delta+\mu(8r)^2)^{\frac{2-p}{4}}$. The
corresponding version of Lemma \ref{lemma:cacci:k} for $(X_j
u-k)^-$, see Remark \ref{remarkadd}, shows that $X_ju$ also belongs to $DG^-(B_{2r})$. So,
$X_ju\in DG(B_{2r})$. We are now in the position to apply Theorem
\ref{DGHolder} to conclude that \begin{equation}\label{est:osc}
\osc_{B_r} X_j u\le (1-2^{s_0})\osc_{B_{2r}} X_j u+ cK^{\frac 1
2}(\delta+\mu(8r)^2)^{\frac{2-p}{4}} r^{\frac 1 2}
\end{equation}
for some $s_0=s_0(n,p,L)>0$. Now taking into account of the
assumptions $p\ge 2$ and (\ref{assum:mu}), we have
\[ (\delta+\mu(8r)^2)^{\frac{2-p}{4}}\le 2^{\frac{p-2}{2}}\omega(r)^{\frac{2-p}{2}}\le 2^{\frac{p-2}{2}}
(\delta+\mu(r_0)^2)^{\frac{2-p}{4}}
\left(\frac{r}{r_0}\right)^{\frac{2-p}{2p}},\]
where in the first inequality we used the fact that
$\mu(8r)\ge \omega(8r)/2\ge \omega(r)/2$,  which follows easily from the
definitions of $\mu$ and $\omega$.
Thus
\begin{equation}\label{estadd10}
cK^{\frac 1
2}(\delta+\mu(8r)^2)^{\frac{2-p}{4}} r^{\frac 1 2}\le c(\delta+\mu(r_0)^2)^{\frac{1}{2}}\left(\frac{r}{r_0}\right)^\frac{1}{p}.
\end{equation}
Now (\ref{est:osc})
becomes
\[
\osc_{B_r} X_j u\le (1-2^{s_0})\osc_{B_{2r}} X_j
u+c(\delta+\mu(r_0)^2)^{\frac{1}{2}}\left(\frac{r}{r_0}\right)^\frac{1}{p}.\]
This proves (\ref{mur4r}). Theorem
\ref{thm:holder2} is proved in this case, modulo the proof of
Lemma \ref{lemma:tau}.

{\it Case 2.} If {\it Case 1} does not happen, then we have for
all $l\in \{ 1,2,\ldots,2n\}$
\begin{equation}\label{assum:tau3}
\vert\{ x\in B_{4r}: X_l u>-\frac{1}{8}\mu(8r)\}\vert>\tau
r^Q,\end{equation} and \begin{equation}\label{assum:tau4} \vert\{
x\in B_{4r}: X_lu<\frac{1}{8}\mu(8r)\}\vert>\tau
r^Q,\end{equation} with $\tau=\tau(n,p,L)>0$, as determined in
Lemma \ref{lemma:tau}. We will prove ({\ref{mur4r}) also in this
case.

First, note that on the set $\{ x\in B_{8r}: X_l
u>\frac{1}{8}\mu(8r)\}$, we have
\begin{equation}\label{Xu:uniform2}
c_3 (\delta+\mu(8r)^2)^{\frac{p-2}{2}}\le (\delta+\vert \X
u\vert^2)^{\frac{p-2}{2}}\le c_4(\delta+\mu(8r)^2)^{\frac{p-2}{2}}
\end{equation}
for some constants $0<c_3<c_4$, depending only on $n,p$. We will
apply Lemma \ref{lemma:cacci:k} with $q=2Q$ for all $k\ge
k_0=\frac{1}{8}\mu(8r)$ in balls $B_{r^{\prime\prime}},
B_{r^\prime}, 0<r^{\prime\prime}<r^\prime\le 8r$. Due to
(\ref{Xu:uniform2}), (\ref{cacci:Xu:k}) becomes
\begin{equation*}
\begin{aligned}
\int_{A_{k,r^{\prime\prime}}^+}\vert\X X_lu\vert^2\, dx\le
&\frac{c}{(r^\prime-r^{\prime\prime})^2}\int_{A_{k,r^\prime}^+}\vert(X_lu-k)^+\vert^2\,
dx\\
&+cK(\delta+\mu(8r)^2)^{\frac{2-p}{2}}\vert
A_{k,r^\prime}^+\vert^{1-\frac 1 Q},
\end{aligned}
\end{equation*}
where $K$
is as in Lemma \ref{lemma:cacci:k}. The above inequality is true for all $l\in \{ 1,2,\ldots, 2n\}$
and for all $k\ge k_0$. We are now in the position to apply Lemma
\ref{DGosc}. Due to (\ref{assum:tau4}), the assumption of Lemma
\ref{DGosc} is satisfied. We conclude that
\begin{equation*}
\begin{aligned}
\sup_{B_{2r}}X_lu
&\le \sup_{B_{8r}}X_lu-2^{-s_0}(\sup_{B_{8r}}
X_lu-\frac{1}{8}\mu(8r))+cK^\frac{1}{2}(\delta+\mu(8r)^2)^{\frac{2-p}{4}}r^{\frac 1 2}\\
&\le \sup_{B_{8r}}X_lu-2^{-s_0}(\sup_{B_{8r}}
X_lu-\frac{1}{8}\mu(8r))+c(\delta+\mu(r_0)^2)^{\frac 1
2}\left(\frac{r}{r_0}\right)^{\frac 1 p}, 
\end{aligned}
\end{equation*}
for some
$s_0=s_0(n,p,L)>0$, where in the second inequality we used (\ref{estadd10}). From (\ref{assum:tau3}), we can derive
similarly, see Remark \ref{remarkadd},  that for
all $l$,
\[ \inf_{B_{2r}}X_l u\ge \inf_{B_{8r}}X_l
u+2^{-s_0}(-\inf_{B_{8r}}X_lu-\frac{1}{8}\mu(8r))-c(\delta+\mu(r_0)^2)^{\frac
1 2}\left(\frac{r}{r_0}\right)^{\frac 1 p}. \] Now the above two
inequalities yield
\[ \osc_{B_{2r}}X_l u\le
(1-2^{-s_0})\osc_{B_{8r}}X_lu+2^{-s_0-2}\mu(8r)+c(\delta+\mu(r_0)^2)^{\frac
1 2}\left(\frac{r}{r_0}\right)^{\frac 1 p}, \] 
and hence
\[ \omega(2r)\le (1-2^{-s_0})\omega(8r)+2^{-s_0-2}\mu(8r)+c(\delta+\mu(r_0)^2)^{\frac
1 2}\left(\frac{r}{r_0}\right)^{\frac 1 p}.\] 
We notice from the conditions (\ref{assum:tau3}) and (\ref{assum:tau4})
that
\[ \omega(8r)\ge \frac{7}{8}\mu(8r).\]
Then the above two inequalities yield
\[ \omega(2r)\le (1-2^{-s_0-1})\omega(8r)+c(\delta+\mu(r_0)^2)^{\frac
1 2}\left(\frac{r}{r_0}\right)^{\frac 1 p}.\] 
This proves
(\ref{mur4r}), and
therefore Theorem \ref{thm:holder2} in this case. The proof of
Theorem \ref{thm:holder2} is complete, modulo the proof of Lemma
\ref{lemma:tau}.

\begin{proof}[Proof of Lemma \ref{lemma:tau}]
We will use Lemma \ref{DGvanish} to prove Lemma \ref{lemma:tau}.
Suppose that (\ref{assum:tau}) holds for some $l\in \{1,2,\ldots,
2n\}$. The case that (\ref{assum:tau2}) holds can be handled
similarly.

Lemma \ref{lemma:cacci:k} with $q=2Q$ yields for all $l$, all
$k\in {\mathbb R}$, and all balls $B_{r^{\prime\prime}},
B_{r^\prime}, 0<r^{\prime\prime}<r^\prime\le 4r$,
\begin{equation}\label{XU100}
\begin{aligned}
&\int_{A_{k,r^{\prime\prime}}^+}\weight\vert\X X_lu\vert^2\, dx\\
\le&
\frac{c}{(r^\prime-r^{\prime\prime})^2}\int_{A_{k,r^\prime}^+}\weight\vert(X_lu-k)^+\vert^2\,
dx
+c K \vert A_{k,r^\prime}^+\vert^{1-\frac 1 Q},
\end{aligned}
\end{equation}
where $K=r_0^{-1}(\delta+\mu(r_0)^2)^{\frac p 2}.$ Now, we denote
$v=(\delta+\vert X_lu\vert^2)^{\frac{p-2}{4}}X_lu$. 
Note that
\[ \weight\vert \X X_lu\vert^2\ge \frac{4}{p^2}\vert \X v\vert^2.\]
Thus we can rewrite (\ref{XU100}) as
\begin{equation}\label{Xupower}
\int_{E_{h,r^{\prime\prime}}^+}\vert \X v\vert^2\, dx\le \frac{
cM(r)^2}{(r^\prime-r^{\prime\prime})^2} \vert
E_{h,r^\prime}^+\vert+\chi^2\vert E_{h,r^\prime}^+\vert^{1-\frac 1
Q},
\end{equation}
where $E_{h,\rho}^+=\{ x\in B_\rho: v(x)>h\}$ and
\[ M(r)=2c(\delta+\mu(8r)^2)^{\frac{p-2}{4}}\mu(8r), \quad
\chi=2^{-\frac p 2}cr_0^{-\frac 1 2}(\delta+\mu(r_0)^2)^\frac{p}{4}.\] Thus
$v\in SDG^+(B_{4r})$. We can now apply Lemma \ref{DGvanish} for a small constant $b>0$, to be chosen soon. Let
$k_0=-\frac{1}{8}\mu(8r)$ and
$h_0=(\delta+k_0^2)^{\frac{p-2}{4}}k_0$. We have
\begin{equation}\label{supv}
\sup_{B_{2r}} v\le h_0+b M(r),
\end{equation}
provided that
\begin{equation}\label{hypo1}
\vert E_{h_0, 4r}^+\vert \le \theta(4r)^Q
\end{equation}
and that
\begin{equation}\label{hypo2}
M(r)\ge \chi (4r)^{\frac 1 2}.
\end{equation}
Note that if we choose $b=b(n,p,L)>0$ small enough, the conclusion
(\ref{supv}) yields
\[ \sup_{B_{2r}} v\le
-\left(\delta+\big(\frac{\mu(8r)}{16}\big)^2\right)^{\frac{p-2}{4}}\frac{\mu(8r)}{16},\]
which is equivalent to the desired result in the lemma
\[ \sup_{B_{2r}} X_l u\le -\frac{1}{16}\mu(8r).\]
Fix such a number $b$. We check the assumptions (\ref{hypo1}) and
(\ref{hypo2}). Observe that (\ref{hypo2}) follows from
(\ref{assum:mu}) and from the fact that
$\mu(8r)\ge \omega(r)/2$, since
\[ M(r)\ge 2^{1-\frac p 2}c\omega(r)^{\frac p 2}\ge 2^{1-\frac p 2}c \big(\delta+\mu(r_0)^2\big)^{\frac p
4}\left(\frac {r}{r_0}\right)^{\frac 1 2}=\chi (4r)^{\frac 1 2}.\]
We also observe that
\[ E_{h_0,4r}^+=\{ x\in B_{4r}: X_lu >-\frac{1}{8}\mu(8r)\}.\]
So, if we choose $\tau=4^Q\theta$, then (\ref{assum:tau}) is
equivalent to (\ref{hypo1}). Let $\theta$ be as in Lemma
\ref{DGvanish}. Then Lemma \ref{lemma:tau} is true for
$\tau=4^Q\theta$. This finishes the proof of Lemma
\ref{lemma:tau}.
\end{proof}

\section{Approximation}\label{section:approx}
In Section \ref{section:lip} and \ref{section:Holder}, we proved Theorem \ref{thm:lip} and Theorem \ref{thm:holder}
for the case $\delta>0$, under the additional assumption that the weak solutions are Lipschitz continuous.
We will remove this additional assumption and prove Theorem \ref{thm:lip} and Theorem \ref{thm:holder}
for the case $\delta>0$ in subsection \ref{subsection:proofpositive}. The proof is based on the Hilbert-Haar existence theory for
functional (\ref{functional}) in the setting of Heisenberg group, which is established in subsection \ref{subsection:Hilbert}.
Finally, in subsection \ref{subsection:proofzero} we prove Theorem \ref{thm:lip} and Theorem \ref{thm:holder} for the case $\delta=0$,
by an approximation argument.

\subsection{Hilbert-Haar existence theory}\label{subsection:Hilbert}
In the setting of Euclidean spaces, the Hilbert-Haar theory gives the existence of Lipschitz continuous minimizers for convex functional
with smooth boundary value in strictly convex domain. In this subsection, we establish the analogous existence theory 
for functional (\ref{functional}) in the
setting of Heisenberg group. Our approach is quite similar to that in the Euclidean setting.

Let $D\subset {\mathbb H}^n$ be a bounded domain. We consider $D$ to be a domain in ${\mathbb R}^{2n+1}$. We assume that $D$ is a convex domain
in ${\mathbb R}^{2n+1}$, and that there exists a constant $\varepsilon>0$ such that the following holds: for every
$y\in \partial D$, there is a vector $b\in {\mathbb R}^{2n+1}$ with $\vert b \vert=1$ such that
\begin{equation}\label{convex}
\langle x-y, b\rangle \ge \varepsilon \vert x-y\vert^2, \quad \forall\, x\in D.
\end{equation}
Here $\langle\cdot,\cdot\rangle$ denotes the inner product of two vectors in ${\mathbb R}^{2n+1}$ and $\vert\cdot\vert$ the Eclidean norm.
The following theorem gives the existence of Lipschitz continuous solutions with $C^2$ smooth boundary value in such domains $D$.

\begin{theorem}\label{thm:hilbert}
Suppose that $D\subset {\mathbb R}^n$ is a bounded domain satisfying condition (\ref{convex}). Let $\phi$ be a $C^2$ function in a
neiborhood $D^\prime$ of $D$, and $u\in HW^{1,p}(D)$ be the unique solution of
\begin{equation}\label{direchelet}
\begin{cases}
\divo (Df(\X u))=0 \quad & \textrm{ in } D\\
u-\phi\in HW^{1,p}_0(D) & \textrm{ on } \partial D
\end{cases}
\end{equation}
Then there is a constant $M$, depending only on $n,\varepsilon, \max_{\overline D}(\vert D\phi\vert+\vert D^2\phi\vert)$ and the diameter of $D$, such that
\[ \norm \X u\norm_{L^\infty(D)}\le M.\]
\end{theorem}

\begin{proof}
We will show that $u$ is Lipschitz continuous in $D$ with repect to the Carnot-Carath\`eodory metric, that is,
\begin{equation}\label{uislip}
\vert u(x)-u(y)\vert\le M d(x,y), \quad \forall\, x,y\in \overline D,
\end{equation}
for a constant $M$, depending on $n,\varepsilon, \max_{\overline D}(\vert D\phi\vert+\vert D^2\phi\vert)$ and the diameter of $D$. Then the conclusion of Theorem \ref{thm:hilbert} follows.

First, we prove (\ref{uislip}) for all $y\in \partial D$ and all $x \in \overline D$. To this end, we fix a point $y\in \partial D$ and $b\in {\mathbb R}^{2n+1}$ be
the vector satisfying condition (\ref{convex}). The essential point of the proof is to consider the
following linear function
\[ L^+(x)=\phi(y)+\langle D\phi(y)+Kb, x-y\rangle\]
for a big constant $K>0$, to be determined soon. One good point to consider this function is that we have
\[ \phi(x)\le L^+(x), \quad \forall \, x\in \partial D.\]
Indeed, by the Taylor formular, we have
\[ \phi(x)=\phi(y)+\langle D\phi(y),x-y\rangle+\frac{1}{2}\sum_{i,j=1}^{2n+1} D_iD_j \phi(\xi)(x_i-y_i)(x_j-y_j),\]
where $\xi$ is a convenient point between $x=(x_1,\ldots, x_{2n+1})$ and $y=(y_1,\ldots,y_{2n+1})$.
Thus, if we let $K=\frac{(2n+1)^2}{2\varepsilon}\max_{\overline D}\vert D^2\phi\vert$, (\ref{convex}) implies that
\[ \phi(x)\le \phi(y)+\langle D\phi(y),x-y\rangle+K\varepsilon\vert x-y\vert^2\le L^+(x)\]
for all $x\in \overline D$, in particular for all $x\in \partial D$.

Another good point to consider $L^+(x)$ is that it is a solution of equation (\ref{equation:main}) in $D$. This is easy to check. Indeed, since $L^+(x)$ is a smooth
function, we can write
\[\divo (Df(\X L^+))=\langle D^2 f(\X L^+), \X\X L^+\rangle= \sum_{i,j=1}^{2n}D_iD_j f(\X L^+)X_iX_j L^+.\]
By a direct calculation, we can show that the horizontal Hessian matrix $\X \X L^+$ is anti-symmetric. We have for all $i=1,2,\dots,2n$,
$X_iX_{n+i}L^+(x)=-X_{n+i}X_iL^+(x)$ and $X_iX_j L^+(x)=X_jX_iL^+(x)=0$ if $j\neq n+i$.
Note that the Hessian matrix $D^2 f$ is symmetric. Thus, the inner product of $D^2f$ and $\X\X L^+$ is zero. This shows that
$L^+(x)$ is a solution of equation (\ref{equation:main}).

We claim that $u(x) \le L^+(x)$ for all $x\in \overline D$, where $u$ is the solution of equation (\ref{direchelet}).
This follows from the comparison principle. Since $D$ is a regularity domain and the boundary value $\phi$ is continuous,
we know that the solution $u$ of equation (\ref{direchelet}) is continuous up to the boundary and $u(x) =\phi(x)$ for all $x\in \partial D$.
Thus, we have $u(x)=\phi(x)\le L^+(x)$ for all $x\in \partial D$. Since $u$ and $L^+$ are both solutions of equation (\ref{equation:main}), the claim follows from the
comparison principle.
Now, we observe that $L^+(x)$ is Lipschitz continuous, that is,
\[ \vert L^+(x)-L^+(z)\vert\le M d(x,z), \quad \forall \, x,z\in \overline D\]
for a constant $M>0$, depending on $n,\varepsilon, \max_{\overline D}(\vert D\phi\vert+\vert D^2\phi\vert)$ and the diameter of $D$. We then have
\[ u(x)-u(y)\le L^+(x)-u(y)=L^+(x)-L^+(y)\le M d(x,y)\]
Similarly, we can define
\[ L^-(x)=\phi(y)+\langle D\phi(y)-Kb, x-y\rangle\]
and show that
\[ u(x)-u(y)\ge -\tilde M\vert x-y\vert.\]
This proves (\ref{uislip}) for all $y\in \partial D$ and all $x\in \overline D$.

To prove (\ref{uislip}) for all $x,y\in \overline D$, we will show the following maximum princilpe
for the differential ratio (John von Neumann theorem):
\[ \sup_{x,y\in \overline D}\frac{\vert u(x)-u(y)\vert}{d(x,y)}=\sup_{x\in \overline D, y\in \partial D}\frac{\vert u(x)-u(y)\vert}{d(x,y)}.\]
Indeed, for any $z\in {\mathbb H}^n$, such that $D_z\cap D\neq \emptyset$, where $D_z=\{zx: x\in D\}$. Set
\[ C=\sup_{x\in \partial (D_z\cap D)} (u(zx)-u(x)).\]
Obviously, we have $u(zx)\le u(x)+C$ for all $x\in \partial(D_z\cap D)$. On the other hand,
the functions $u(zx)$ and $u(x)+C$ are weak solutions of equation (\ref{equation:main}) in $D_z$ and $D$, respectively. Thus they are weak solutions
in $D_z\cap D$. Now the comparison principle implies that $u(zx)\le u(x)+C$ for all $x\in D_z\cap D$, that is,
\[ u(zx)-u(x)\le \sup_{w\in \partial (D_z\cap D)} (u(zw)-u(w)), \quad \forall \, x\in D_z\cap D.\]
We rewrite the above inequality as
\[ \frac{u(zx)-u(x)}{d(z,0)}\le \sup_{w\in \partial (D_z\cap D)} \frac{u(zw)-u(w)}{d(z,0)}.\]
Since $w\in\partial (D_z\cap D)$ implies that $w\in \partial D$ or $zw\in \partial D$,  the above inequality, together with the arbitrariness of $z$, implies
the maximun princilpe for the differential ratio. This proves the theorem.
\end{proof}

\subsection{Proofs of Theorem \ref{thm:lip} and Theorem \ref{thm:holder} for the case $\delta>0$}\label{subsection:proofpositive}
Let $u\in HW^{1,p}(\Omega)$ be a weak solution of equation (\ref{equation:main}).
We fix a ball $B_{r}(x_0)\subset \Omega$. We may assume that $x_0=0$, and we write $B_{r}=B_{r}(0)$.
Since $C^\infty(B_{r})$ is dense in $HW^{1,p}(B_{r})$, we can find a sequence of functions
$\phi_\varepsilon\in C^\infty(B_{r})$ such that $\phi_\epsilon$ converges to $u$ in $HW^{1,p}(B_{r})$.
Now, since the Carnot-Carath\`eodory metric and the  Kor\`anyi metric are equivalent, we can find a  Kor\`anyi ball $K_{\sigma r}\Subset B_{r}$, centered at $0$,
with a constant $\sigma=\sigma(n)>0$. The reason that we consider the  Kor\`anyi ball $K_{\sigma r}$ is that it is convex (in ${\mathbb R}^{2n+1}$) and
it satisfies condition (\ref{convex}). Now we solve the Direchlet problem
\begin{equation}\label{direchelet2}
\begin{cases}
\divo (Df(\X v))=0 \quad & \textrm{ in } K_{\sigma r}\\
v-\phi_\varepsilon\in HW^{1,p}_0( K_{\sigma r}) & 
\end{cases}
\end{equation}
Let $u_\varepsilon$ be the unique weak solution of the above equation. We use $\varphi=u_\varepsilon-\phi_\varepsilon$ as a test-function in
the above equation to obtain that
\[ \int_{K_{\sigma r}} \langle Df(\X u_\varepsilon), \X u_\varepsilon\rangle\, dx=\int_{K_{\sigma r}}\langle Df(\X u_\varepsilon), \X\phi_\varepsilon\rangle\, dx.\]
Then the ellipticity condition (\ref{stru})  and the structure condition (\ref{structure}) yields
\[ \int_{K_{\sigma r}}\big(\delta+\vert\X u_\varepsilon\vert^2\big)^{\frac{p-2}{2}}\vert\X u_\varepsilon\vert^2\, dx
\le c\int_{K_{\sigma r}}\big(\delta+\vert\X u_\varepsilon\vert^2\big)^{\frac{p-1}{2}}\vert\X\phi_\varepsilon\vert+c\delta^{\frac p 2}\vert K_{\sigma r}\vert,\]
which implies that
\begin{equation}\label{uepsilon}
\int_{K_{\sigma r}}\big(\delta+\vert\X u_\varepsilon\vert^2\big)^{\frac p 2}\, dx\le c\int_{K_{\sigma r}}\big(\delta+\vert\X \phi_\varepsilon\vert^2\big)^{\frac p 2}\, dx
\le c\int_{K_{\sigma r}}\big(\delta+\vert\X u\vert^2\big)^{\frac{p}{2}}\, dx+o(\varepsilon),
\end{equation}
where $c=c(p, L)>0$ and $o(\varepsilon)\to 0$ as $\varepsilon\to 0$. Since $\phi_\varepsilon$ is smooth and $K_{\sigma r}$ satisfies the condition
(\ref{convex}), we may apply Theorem \ref{thm:hilbert} to conclude that
\[\norm \X u_\varepsilon\norm\le M\]
for a constant $M>0$, depending on $n, K_{\sigma r}, \max_{K_{\sigma r}}(\vert D\phi_\varepsilon\vert+\vert D^2\phi_\varepsilon\vert)$.
Let $B_{2\tau r}$ be a ball in $K_{\sigma r}$, centered at 0, with a constant $\tau=\tau(n)>0$.
Now, we can apply Theorem \ref{thm:lip} and Theorem \ref{thm:holder} to the weak solution $u_\varepsilon$, since it is Lipschitz continuous. We have
that
\[ \sup_{B_{\tau r}}\vert\X u_\varepsilon\vert\le c\left(\intav_{B_{2\tau r}}\big(\delta+\vert\X u_\varepsilon\vert^2\big)^{\frac{p}{2}}\, dx\right)^{\frac 1 p},\]
if $1<p<\infty$, and that for all $0<\rho\le \tau r$,
\[ \max_{1\le l\le 2n}\osc_{B_\rho} X_lu_\varepsilon\le c\big(\frac{\rho}{r}\big)^\alpha
\left(\intav_{B_{2\tau r}}\big(\delta+\vert\X u_\varepsilon\vert^2\big)^{\frac{p}{2}}\, dx\right)^{\frac 1 p},\]
if $2\le p<\infty$. Here $c=c(n,p,L)>0$, $\tau=\tau(n)>0$ and $\alpha=\alpha(n,p,L)>0$. They do not depend on $\varepsilon, \delta$ or $r$.
Now we let $\varepsilon\to 0$ to conclude the proof. (\ref{uepsilon}) implies that $u_\varepsilon$ converges weakly to a function $\bar u$ in $HW^{1,p}(K_{\sigma r})$.
On one hand, we have $u_\varepsilon-\phi_\varepsilon\in  HW^{1,p}_0( K_{\sigma r})$, which means that $\bar u-u\in  HW^{1,p}_0( K_{\sigma r})$. On the other hand,
it is easy to show that $\bar u$ is a weak solution of equation (\ref{equation:main}). Now the uniqueness of solution of equation (\ref{equation:main}) implies that
$\bar u=u$. We then conclude from the above two estimates that
if $1<p<\infty$,
\[ \sup_{B_{\tau r}}\vert\X u\vert\le c\left(\intav_{B_{r}}\big(\delta+\vert\X u\vert^2\big)^{\frac{p}{2}}\, dx\right)^{\frac 1 p},\]
and that if $2\le p<\infty$,
\[ \max_{1\le l\le 2n}\osc_{B_\rho} X_lu\le c\big(\frac{\rho}{r}\big)^\alpha
\left(\intav_{B_{r}}\big(\delta+\vert\X u\vert^2\big)^{\frac{p}{2}}\, dx\right)^{\frac 1 p},\]
for all $0<\rho\le \tau r$. These estimates holds for all $B_r\subset \Omega$. Now Theorem \ref{thm:lip}
and Theorem \ref{thm:holder} follows from a simple covering argument. This completes the proofs.

\subsection{Proofs of Theorem \ref{thm:lip} and Theorem \ref{thm:holder} for the case $\delta=0$}\label{subsection:proofzero}
The proof follows from an easy approximation argument. We only mention the line of the proof.
Suppose that the integrand function $f$ of functional (\ref{functional}) satisfies the structure condition
\begin{equation}\label{structure3}
\begin{aligned} \vert z\vert^{p-2} \vert\xi\vert^2\le
\langle D^2f(z) &\xi,\xi\rangle \le \vert
z\vert^{p-2}\vert \xi\vert^2;\\
\vert Df(z)\vert & \le \vert z\vert^{p-1},
\end{aligned}
\end{equation}
We may assume that
$f(0)=0$. For $\delta>0$, we define the function 
\begin{equation}\label{fdelta}
f_\delta(z)=\begin{cases} \big(\delta+f(z)^{\frac{2}{p}}\big)^{\frac p 2}, \quad &\text{ if } 1<p<2;\\
\delta^{\frac{p-2}{2}}\vert z\vert^2+f(z), &\text{ if } p\ge 2.
\end{cases}
\end{equation}
It is not hard to check that the new function $f_\delta$ satisfies (\ref{structure}) with $\delta>0$, that is,
\begin{equation}\label{structure4}
\begin{aligned} \frac{1}{\overline L}(\delta+\vert z\vert^2)^{\frac{p-2}{2}} \vert\xi\vert^2\le
\langle D^2f_\delta (z) &\xi,\xi\rangle \le {\overline L}(\delta+\vert
z\vert^2)^{\frac{p-2}{2}}\vert \xi\vert^2;\\
\vert Df_\delta(z)\vert & \le {\overline L}(\delta+\vert z\vert^2)^{\frac{p-1}{2}},
\end{aligned}
\end{equation}
where ${\overline L}={\overline L} (p,L)>0$.
Now let $u$ be a solution of equation (\ref{equation:main}) satisfying (\ref{structure3}). We let $u_\delta\in HW^{1,p}(\Omega)$
be the unique solution of 
\begin{equation}\label{direchelet3}
\begin{cases}
\divo (Df_\delta(\X v))=0 \quad & \textrm{ in } \Omega\\
v-u\in HW^{1,p}_0( \Omega) & \textrm{ on } \partial  \Omega.
\end{cases}
\end{equation}
Then we may apply Theorem \ref{thm:lip} and Theorem \ref{thm:holder} to the weak solutions $u_\delta$. 
We can obtain the uniform estimates in these theorems. Eventually, letting $\delta\to 0$, we conclude the proof.    
\section{Appendix}

\begin{proof}[Proof of Lemma \ref{caccioppoli:T}]
Let $\eta\in C^\infty_0(\Omega)$ and $\varphi=\eta^2\vert
Tu\vert^\beta Tu$. We use $\varphi$ as a test-function in the
equation (\ref{equation:T}) for  $Tu$ to obtain that
\begin{equation*}
\begin{aligned}
(\beta+1)\int_\Omega \sum_{i,j=1}^{2n}&\eta^2\vert Tu\vert^\beta D_jD_if(\Xu)X_jTuX_iTu\, dx\\
&=-2\int_\Omega \sum_{i,j=1}^{2n}\eta\vert Tu\vert^\beta TuD_jD_if(\Xu)X_jTuX_i\eta\, dx.
\end{aligned}
\end{equation*}
By the structure condition (\ref{structure}), we have
\begin{equation*}
\begin{aligned}
&\int_\Omega \eta^2\weight\vert Tu\vert^\beta\vert \X Tu\vert^2\, dx\\
\le& \frac{c}{\beta+1}\int_\Omega \vert\eta \vert\vert\X\eta\vert\weight\vert
Tu\vert^{\beta+1}\vert\X Tu\vert\, dx,
\end{aligned}
\end{equation*}
from which, together with the Cauchy-Schwarz inequality, the lemma follows.
\end{proof}

\begin{proof}[Proof of Lemma \ref{caccioppoli:horizantal:sigma}]
Fix $l\in \{ 1, 2,\ldots,n\}$, $\beta\ge 0$ and $\eta\in C^\infty_0(\Omega)$.
let $\varphi=\eta^2\deltaX^{\frac \beta 2}X_l u$. We use $\varphi$
as a test-function in the 
equation (\ref{equation:horizontal}) to obtain that
\begin{equation}\label{weak5}
\begin{aligned}
\int_\Omega \sum_{i,j=1}^{2n}&\eta^2D_jD_i f(\Xu)X_jX_l uX_i\big(\deltaX^{\frac \beta 2}X_l u\big)\, dx\\
=&-2\int_\Omega\sum_{i,j=1}^{2n}\eta D_jD_if(\Xu)X_jX_l uX_i\eta\deltaX^{\frac \beta 2}X_lu\, dx\\
&-\int_\Omega \sum_{i=1}^{2n} D_{n+l}D_i f(\Xu)TuX_i\big(\eta^2\deltaX^{\frac\beta 2}X_l u\big)\, dx\\
&+\int_\Omega T\big(D_{n+l}f(\Xu)\big)\big(\eta^2\deltaX^{\frac\beta 2}X_l u\big)\, dx\\
=& I_1^l+I_2^l+I_3^l.
\end{aligned}
\end{equation}
Similarly, we can deduce from equation (\ref{equation:horizontal2}) that for all $l\in \{ n+1,\ldots,2n\}$,
\begin{equation}\label{weak6}
\begin{aligned}
\int_\Omega \sum_{i,j=1}^{2n}&\eta^2D_jD_i f(\Xu)X_jX_l uX_i\big(\deltaX^{\frac \beta 2}X_l u\big)\, dx\\
=&-2\int_\Omega\sum_{i,j=1}^{2n}\eta D_jD_if(\Xu)X_jX_l uX_i\eta\deltaX^{\frac \beta 2}X_lu\, dx\\
&+\int_\Omega \sum_{i=1}^{2n} D_{l-n}D_i f(\Xu)TuX_i\big(\eta^2\deltaX^{\frac\beta 2}X_l u\big)\, dx\\
&-\int_\Omega T\big(D_{l-n}f(\Xu)\big)\big(\eta^2\deltaX^{\frac\beta 2}X_l u\big)\, dx\\
=& I_1^l+I_2^l+I_3^l.
\end{aligned}
\end{equation}
Summing the above equations for all $l$ from $1$ to $2n$, we arrive at
\begin{equation}\label{weak7}
\int_\Omega \sum_{i,j, l} \eta^2D_jD_i f(\Xu)X_jX_l uX_i
\left(\deltaX^{\frac \beta 2}X_l u\right)\, dx=\sum_{l} \left(I_1^l+I_2^l+I_3^l\right).
\end{equation}
All sums for $i,j,k$ are from 1 to $2n$. We will estimate both
sides of (\ref{weak7}) as follows. For the left hand side, we note
that
\[ X_i\left(\deltaX^{\frac \beta 2} X_l u\right)=
\deltaX^{\frac \beta 2}X_iX_l u+\frac{\beta}{2} \deltaX^{\frac{\beta-2}{2}} X_i(\vert\Xu\vert^2)X_l u.
\]
Thus, by the structure condition (\ref{structure}),
\[
\sum_{i,j,l}D_jD_i f(\Xu)  X_jX_l uX_i \left(\deltaX^{\frac \beta 2} X_l u\right)\ge
\deltaX^{\frac{p-2+\beta}{2}}\vert\X \X u\vert^2.
\]
Hence,
\begin{equation}\label{weak8}
\text{left of (\ref{weak7})} \ge \int_\Omega \eta^2\deltaX^{\frac{p-2+\beta}{2}}\vert\X\X u\vert^2\, dx.
\end{equation}
For the right hand side, we will show that for each 
$l=1,\ldots,2n$ and each $m=1,2,3$, we have
\begin{equation}\label{claim100}
\begin{aligned}
\vert I_m^l\vert\le & \frac{1}{12n}
\int_\Omega\eta^2\deltaX^{\frac{p-2+\beta}{2}}\vert \X \X u \vert^2\, dx\\
&+c(\beta+1)^4\int_\Omega \eta^2\deltaX^{\frac{p-2+\beta}{2}}\vert Tu\vert^2\, dx\\
&+c\int_\Omega (\vert\X\eta\vert^2+\vert\eta\vert\vert
T\eta\vert)\deltaX^{\frac{p+\beta}{2}}\, dx,
\end{aligned}
\end{equation}
where $c=c(n,p,L)>0$.
Then the lemma follows from the above estimates for
both sides of (\ref{weak7}). This completes the proof the lemma,
modulo the proof of (\ref{claim100}).

We now prove (\ref{claim100}). First, to prove that
(\ref{claim100}) holds for for $I_1^l, l=1,2,\ldots, 2n$, we have
by the structure condition (\ref{structure})
\begin{equation}\label{weak9}
\vert I_1^l\vert \le c\int_\Omega \vert\eta\vert\vert\X
\eta\vert\deltaX^{\frac{p-1+\beta}{2}}\vert\X \X u\vert\, dx,
\end{equation}
from which, together with the Cauchy-Schwarz inequality, (\ref{claim100}) for
$I_1^l$ follows. Second, for $I_2^l, l=1,2,\ldots, 2n$, we have
\begin{equation}\label{weak10}
\begin{aligned}
\vert I^l_2\vert\le &c(\beta+1)\int_\Omega \eta^2 \deltaX^{\frac{p-2+\beta}{2}}\vert Tu\vert \vert\X\X u\vert\, dx\\
&+c\int_\Omega \vert\eta\vert\vert\X
\eta\vert\deltaX^{\frac{p-1+\beta}{2}}\vert Tu\vert\, dx,
\end{aligned}
\end{equation}
from which, together with the Cauchy-Schwarz inequality, (\ref{claim100}) for
$I_2^l$ follows, too. Finally, for $I^l_3, l=1,2,\ldots, n$,
integration by part yields
\begin{equation*}
\begin{aligned}
I_3^l=&-\int_\Omega D_{n+l}f(\X u)T\big(\eta^2\deltaX)^{\frac{\beta}{2}} X_l u\big)\, dx\\
=&-\int_\Omega\eta^2 D_{n+l}f(\X u)\deltaX^{\frac \beta 2} X_l Tu\, dx\\
&-\beta\int_\Omega \sum_{k=1}^{2n} \eta^2 D_{n+l}f(\X u)\deltaX^{\frac{\beta-2 }{2}}X_ku X_l u X_kTu\, dx\\
&-2\int_\Omega\eta T\eta D_{n+l} f(\X u)\deltaX^{\frac{\beta}{2}}X_lu\, dx.
\end{aligned}
\end{equation*}
Again by integration by parts, we obtain
\begin{equation*}
\begin{aligned}
I_3^l=&\int_\Omega X_l\left(\eta^2 D_{n+l} f(\X u)\deltaX^{\frac{\beta}{2}}\right)Tu\, dx\\
&+\beta\int_\Omega \sum_{k=1}^{2n} X_k\left(\eta^2 D_{n+l}f(\X u)\deltaX^{\frac{\beta-2}{2}}X_kuX_lu\right) Tu\, dx\\
&-2\int_\Omega\eta T\eta D_{n+l}f(\X u)\deltaX^{\frac{\beta}{2}}X_l u\, dx.
\end{aligned}
\end{equation*}
Thus, by the structure condition (\ref{structure})
\begin{equation}\label{weak11}
\begin{aligned}
\vert I_3^l\vert\le & c(\beta+1)^2 \int_\Omega \eta^2 \deltaX^{\frac{p-2+\beta}{2}}\vert Tu\vert\vert\X\X u\vert\, dx\\
&+c(\beta+1)\int_\Omega \vert\eta\vert\vert \X \eta \vert \deltaX^{\frac{p-1+\beta}{2}}\vert Tu\vert\, dx\\
&+c \int_\Omega \vert\eta\vert\vert T\eta\vert\deltaX^{\frac{p+\beta}{2}}\,
dx
\end{aligned}
\end{equation}
from which, together with the Cauchy-Schwarz inequality, (\ref{claim100}) for
$I_3^l, l=1,2,\ldots, n,$ follows, too. 
Similarly, we can prove (\ref{claim100}) for $I_3^l, l=n+1,\ldots, 2n$.
This finishes the proof of (\ref{claim100}),
and hence that of the lemma.
\end{proof}

\end{document}